\documentclass[11pt]{amsart}

\usepackage[a4paper,top=2cm,bottom=2cm,left=2cm,right=2cm]{geometry}

\usepackage[T1]{fontenc}
\usepackage{amsmath,amssymb,amsthm}
\usepackage{graphicx}
\usepackage{caption}
\usepackage{subcaption}
\usepackage{float}
\usepackage{hyperref}
\usepackage{placeins}
\usepackage{amsmath}
\usepackage{float}

\newcommand{\Ds}{\displaystyle}

\newtheorem{lemma}{Lemma}[section]

\newtheorem{theorem}{Theorem}[section]

\newtheorem{remark}{Remark}[section]

\theoremstyle{definition}

\numberwithin{equation}{section}

\title[Periodic-to-Localized Convergence in 2D Saturable DNLS]
{Ground States and Periodic--to--Localized Convergence
in Two--Dimensional Saturable Discrete Nonlinear Schr\"odinger Equations}

\author{Vassilios M Rothos}
\address{School of Mechanical Engineering, Faculty of Engineering\\
Aristotle University of Thessaloniki\\
Thessaloniki 54124, Greece}
\email{rothos@auth.gr}

\thanks{Accepted for publication in \textit{Applicable Analysis}.}
\thanks{
The research was funded by Aristotle University of Thessaloniki (AUTH) Research Council grant numbers 73191, 73699, 11682.}

\subjclass[2020]{35Q55, 37K40, 39A12, 35B35}

\keywords{discrete nonlinear Schr\"odinger equation,
saturable nonlinearity,
ground states,
Nehari manifold,
orbital stability,
strong convergence}
\date{\today}

\dedicatory{}

\begin{document}
\begin{abstract}
We study a two--dimensional discrete nonlinear Schr\"odinger equation with saturable nonlinearity on the lattice $\mathbb Z^2$. Using a variational approach based on the Nehari manifold, we establish the existence of nontrivial periodic ground states on finite lattices and establish the existence of exponentially localized ground states in $\ell^2(\mathbb Z^2)$. A principal result is the rigorous passage from periodic to localized states: we show that, up to lattice translations, periodic ground states converge strongly in $\ell^2(\mathbb Z^2)$ to a localized ground state as the lattice periods tend to infinity. The analysis combines variational methods, spectral properties of the discrete Laplacian, and concentration--compactness techniques adapted to the two--dimensional discrete setting. We further derive qualitative properties of the resulting solutions, including positivity and exponential localization, and establish a conditional orbital stability result within the Grillakis--Shatah--Strauss framework. Numerical computations illustrate the theoretical results and confirm the predicted convergence and localization behavior.

\end{abstract}
\maketitle
\section{Introduction}
\label{sec1}

Discrete nonlinear Schr\"odinger (DNLS) equations constitute a fundamental
class of lattice models describing nonlinear wave propagation in spatially
discrete media. They arise naturally in a variety of physical settings,
including nonlinear optical waveguide arrays, photorefractive crystals,
and Bose--Einstein condensates trapped in optical lattices; see, for
example, the monograph of Kevrekidis \cite{Kevrekidis2009} and the survey
of Kevrekidis and Pelinovsky \cite{KevrekidisPelinovsky2006}.
The interplay between lattice discreteness and nonlinearity leads to the
formation of spatially localized coherent structures, commonly known as
discrete solitons or intrinsic localized modes.

Among the nonlinear responses considered in the DNLS framework,
saturable nonlinearities play a particularly important role.
Unlike polynomial nonlinearities, they remain bounded for large
amplitudes and therefore provide a more realistic description of
photorefractive and optical media.
From the analytical point of view, saturable nonlinearities are
asymptotically linear and nonpolynomial, which makes the associated
variational problems substantially more delicate.
Relevant physical and mathematical studies include
Christodoulides, Lederer and Silberberg \cite{Christodoulides2003},
Had\v{z}ievski et al.~\cite{Hadzievski2004},
Maluckov et al.~\cite{Maluckov2008}, Melvi et al.~\cite{MelvinChampneysKevrekidisCuevas2006,MelvinChampneysKevrekidisCuevas2008} 
and the more recent work of Alfimov et al.~\cite{Alfimov2019}.

Variational methods have proved particularly effective for the study of
localized states in DNLS lattices. 
In the one--dimensional setting, Pankov and Rothos
\cite{PankovRothos2008} developed a Nehari manifold approach for DNLS
equations with saturable nonlinearity and established the existence of
periodic and spatially localized solutions.
Related variational approaches for nonlinear lattice systems may also be
found in Herrmann \cite{Herrmann2010}.
However, a corresponding rigorous theory for two--dimensional saturable
DNLS lattices appears to be largely unexplored.

The purpose of the present work is to extend the variational framework of
\cite{PankovRothos2008} to the two--dimensional lattice $\mathbb Z^2$.
This extension is nontrivial.
Passing from $\mathbb Z$ to $\mathbb Z^2$ enlarges the spectral band of
the discrete Laplacian from $[0,4]$ to $[0,8]$, introduces two
independent translation directions, and significantly complicates the
compactness analysis required for the infinite--lattice limit.
As a consequence, a refined concentration--compactness argument is needed
to control minimizing sequences and exclude loss of mass through lattice
translations.

We consider the two--dimensional DNLS equation with saturable
nonlinearity
\[
i\dot{\psi}_{n,m}
+\Delta\psi_{n,m}
+\frac{\nu|\psi_{n,m}|^2}{1+\mu|\psi_{n,m}|^2}\psi_{n,m}
=0,
\qquad (n,m)\in\mathbb Z^2,
\]
and seek standing wave solutions of the form
$\psi_{n,m}(t)=e^{-i\omega t}u_{n,m}$.
This leads to a stationary lattice equation that admits a natural
variational formulation on $\ell^2(\mathbb Z^2)$.

The analysis is based on the Nehari manifold associated with the action
functional.
Within this framework we first establish the existence of periodic ground
states on finite lattices with arbitrary periods.
We then pass to the infinite--lattice limit and establish the existence of an
exponentially localized ground state on $\mathbb Z^2$.
The principal contribution of the paper is a rigorous periodic-to-localized
convergence theorem showing that, up to lattice translations, periodic
ground states converge strongly in $\ell^2(\mathbb Z^2)$ to a localized
ground state as the lattice periods tend to infinity.
The proof combines variational methods, spectral properties of the
discrete Laplacian, and concentration--compactness techniques adapted to
the two--dimensional lattice setting.

In addition, we establish qualitative properties of the resulting ground
states, including positivity and exponential localization.
We also derive a conditional orbital stability result for the
corresponding standing waves within the
Grillakis--Shatah--Strauss framework \cite{r19,r19a}.
The analytical results are complemented by numerical computations that
illustrate the structure of the ground states and the convergence of
periodic approximations toward the localized infinite--lattice profile.

To the best of our knowledge, this is the first rigorous variational analysis 
establishing periodic--to--localized convergence for a two-dimensional DNLS equation with saturable nonlinearity.

\medskip
\noindent
\textbf{Organization of the paper.}
Section~\ref{sec2} introduces the mathematical formulation of the
problem and the associated variational setting.
Section~\ref{sec3} states the main results.
Section~\ref{sec4} contains the spectral and technical preliminaries
required in the two--dimensional setting.
Sections~\ref{sec5}--\ref{sec7} develop the Nehari manifold framework
and establish the existence of periodic and localized solutions.
Section~\ref{sec8} proves the strong convergence of periodic ground
states.
Section~\ref{sec9} discusses orbital stability.
Finally, Section~\ref{sec10} presents numerical computations
illustrating the analytical results.
\section{Problem set--up}\label{sec2}

We consider the two--dimensional discrete nonlinear Schr\"odinger equation with
saturable nonlinearity posed on the square lattice $\mathbb Z^2$,
\begin{equation}\label{eq:2DDNLS}
i \dot{\psi}_{n,m}
+ \psi_{n+1,m} + \psi_{n-1,m}
+ \psi_{n,m+1} + \psi_{n,m-1}
- 4\psi_{n,m}
+ \frac{\nu |\psi_{n,m}|^2}{1+\mu |\psi_{n,m}|^2}\psi_{n,m}
=0,
\end{equation}
where $(n,m)\in\mathbb Z^2$, $\mu>0$, and $\nu\neq 0$.
Equation \eqref{eq:2DDNLS} arises, for instance, as a tight--binding
approximation describing optical wave propagation in two--dimensional
nonlinear waveguide arrays with saturable response, and constitutes a
prototypical example of a discrete Schr\"odinger equation with
asymptotically linear nonlinearity.

We seek standing wave solutions of the form
\begin{equation}\label{eq:standingwave}
\psi_{n,m}(t)=e^{-i\omega t}u_{n,m},
\end{equation}
where $\omega\in\mathbb R$ and $u_{n,m}\in\mathbb R$.
Substitution of \eqref{eq:standingwave} into \eqref{eq:2DDNLS}
leads to the stationary equation
\begin{equation}\label{eq:stationary2D}
-\Delta u_{n,m}-\omega u_{n,m}=f(u_{n,m}),
\qquad (n,m)\in\mathbb Z^2,
\end{equation}
where the discrete Laplacian is given by
\begin{equation}\label{eq:laplacian}
\Delta u_{n,m}
=
u_{n+1,m}+u_{n-1,m}
+u_{n,m+1}+u_{n,m-1}
-4u_{n,m},
\end{equation}
and the nonlinearity has the saturable form
\begin{equation}\label{eq:saturablenonlinearity}
f(u)=\frac{\nu u^3}{1+\mu u^2}.
\end{equation}

Throughout the paper we work in the Hilbert space
$\ell^2(\mathbb Z^2)$ endowed with the standard inner product
\[
(u,v)
=
\sum_{(n,m)\in\mathbb Z^2}
u_{n,m}v_{n,m},
\]
and norm
\[
\|u\|_2^2
=
\sum_{(n,m)\in\mathbb Z^2}
|u_{n,m}|^2.
\]

The discrete Laplacian $-\Delta$ defines a bounded self--adjoint operator on
$\ell^2(\mathbb Z^2)$ whose spectrum is purely absolutely continuous and given by
\begin{equation}\label{eq:spectrum}
\sigma(-\Delta)=[0,8].
\end{equation}

Accordingly, we introduce the linear operator
\begin{equation}\label{eq:linearoperator}
L=-\Delta-\omega.
\end{equation}
The operator $L$ is invertible whenever
\begin{equation}\label{eq:gapcondition}
\omega\notin[0,8].
\end{equation}
In particular, localized standing waves are expected to exist only for
frequencies belonging to the spectral gaps
$\Ds
\omega\in(-\infty,0)\cup(8,\infty).
$

\begin{remark}\label{rem:staggering}
The two spectral gaps are related through the classical staggering
transformation
\[
u_{n,m}=(-1)^{n+m}v_{n,m},
\]
which maps solutions below the phonon band to solutions above it.
Indeed, substituting this transformation into the stationary equation
shows that frequencies in the semi--infinite gap $(-\infty,0)$ correspond
to frequencies in the upper spectral gap $(8,\infty)$. For this reason,
the analysis developed below is formulated in a unified manner so as to
cover both spectral regimes.
\end{remark}

The nonlinearity $f$ is of class $C^1(\mathbb R)$, odd, and asymptotically linear,
with
\begin{equation}\label{eq:asymptoticlinearity}
\lim_{|u|\to\infty}\frac{f(u)}{u}
=
\frac{\nu}{\mu}
=:l,
\end{equation}
and satisfies
\begin{equation}\label{eq:origin}
f(u)=o(u),
\qquad u\to0.
\end{equation}
Hence the nonlinearity is superlinear near the origin and approaches a
linear response for large amplitudes. These properties place
\eqref{eq:stationary2D} within the class of discrete Schr\"odinger
equations with asymptotically linear nonlinearities and make possible a
variational treatment.

Let
\begin{equation}\label{eq:primitive}
F(s)
=
\int_0^s f(\tau)\,d\tau
=
\frac{\nu}{2\mu^2}
\Big(
\mu s^2-\ln(1+\mu s^2)
\Big).
\end{equation}

For $u$ in the real Hilbert space $\ell^2(\mathbb Z^2)$, we define the associated action functional
\begin{equation}\label{eq:functional}
J(u)
=
\frac12(Lu,u)
-\sum_{(n,m)\in\mathbb Z^2}F(u_{n,m}).
\end{equation}
Critical points of $J$ correspond to $\ell^2(\mathbb Z^2)$ solutions
of the stationary equation \eqref{eq:stationary2D}. Moreover,
$J\in C^1(\ell^2(\mathbb Z^2),\mathbb R)$ and its Euler--Lagrange
equation is precisely \eqref{eq:stationary2D}.

Since $J$ is generally unbounded from above and below on
$\ell^2(\mathbb Z^2)$, we restrict the functional to the associated
Nehari manifold
\begin{equation}\label{eq:nehari}
\mathcal N
=
\Big\{
u\in\ell^2(\mathbb Z^2)\setminus\{0\}
:
\langle J'(u),u\rangle=0
\Big\}.
\end{equation}
The Nehari constraint provides a natural variational setting in which
nontrivial critical points may be characterized as constrained minimizers.

Ground states are defined as minimizers of $J$ on $\mathcal N$.
Their construction proceeds through a sequence of finite periodic lattice
problems. More precisely, for each lattice period $N$ we consider the
corresponding variational problem under periodic boundary conditions,
establish the existence of periodic ground states, and subsequently
investigate the limit as $N\to\infty$.

The principal objective is to show that, under suitable assumptions on
the frequency $\omega$ and on the asymptotic slope
$l=\nu/\mu$, the resulting sequence of periodic ground states converges
to a nontrivial spatially localized solution of
\eqref{eq:stationary2D}. The limiting profile is shown to decay
exponentially and to minimize the infinite--lattice variational problem.
This periodic approximation procedure therefore provides a rigorous
construction of localized ground states on $\mathbb Z^2$ from finite
periodic lattices.

The variational framework introduced above forms the basis for the
analysis developed in the remainder of the paper. In the next section we
state the main results concerning the existence of periodic ground
states, the existence of exponentially localized solutions on the
infinite lattice, and the convergence of periodic approximations as the
lattice periods tend to infinity.
\section{Main results}\label{sec3}

The main results of this paper establish the existence of periodic
ground states on finite periodic lattices, the existence of exponentially
localized ground states on the infinite lattice $\mathbb Z^2$, and a
global convergence result linking these two settings.

More precisely, we establish that:

\begin{itemize}
\item nontrivial periodic ground states exist on finite periodic lattices;
\item a spatially localized ground state exists on the infinite lattice $\mathbb Z^2$;
\item periodic ground states converge, up to lattice translations, to a localized
ground state as the lattice periods tend to infinity.
\end{itemize}

To formulate the results in a general framework, we consider the stationary
lattice equation
\begin{equation}\label{eq:main2D}
-\Delta u_{n,m}-\omega u_{n,m}
=
f(u_{n,m}),
\qquad
(n,m)\in\mathbb Z^2,
\end{equation}
where $\omega\in\mathbb R$ and $f:\mathbb R\to\mathbb R$ satisfies the
following assumptions.

Let
\begin{equation}\label{eq:Fdef}
F(t)
=
\int_0^t f(s)\,ds.
\end{equation}

\begin{itemize}
\item[(h1)]
$f(t)=o(t)$ as $t\to0$;

\item[(h2)]
$\Ds
\lim_{t\to\pm\infty}\frac{f(t)}{t}
=
l
<
\infty;
$

\item[(h3)]
$f\in C^1(\mathbb R)$ and
$\Ds
f(t)t<f'(t)t^2,
$ $t\neq0;$

\item[(h4)]
$\Ds
\frac12f(t)t-F(t)\to\infty,
$ $\Ds
t\to\pm\infty.
$
\end{itemize}

We now introduce the functional framework.
Let $k_1,k_2>1$ be integers and define
\[
Q_{k_1,k_2}
=
\left\{
(n,m)\in\mathbb Z^2:
-\Bigl[\frac{k_1}{2}\Bigr]
\le n\le
k_1-\Bigl[\frac{k_1}{2}\Bigr]-1,
\;
-\Bigl[\frac{k_2}{2}\Bigr]
\le m\le
k_2-\Bigl[\frac{k_2}{2}\Bigr]-1
\right\}.
\]

We denote by $X_{k_1,k_2}$ the space of all
$(k_1,k_2)$--periodic sequences.
Since $Q_{k_1,k_2}$ is finite,
$X_{k_1,k_2}$ is a finite--dimensional Hilbert space endowed with the norm
\[
\|u\|_{k_1,k_2}
=
\left(
\sum_{(n,m)\in Q_{k_1,k_2}}
|u_{n,m}|^2
\right)^{1/2}.
\]

All $\ell^p$ norms are equivalent on $X_{k_1,k_2}$.
On the infinite lattice we consider
$\Ds
X=\ell^2(\mathbb Z^2),
$
equipped with the norm
\[
\|u\|
=
\left(
\sum_{(n,m)\in\mathbb Z^2}
|u_{n,m}|^2
\right)^{1/2}.
\]

Let

\begin{equation}\label{eq:Loperator}
L=-\Delta-\omega.
\end{equation}

We use the same notation for the operator acting either on
$X_{k_1,k_2}$ or on $X$.
In both settings, $L$ is bounded and self--adjoint.

The corresponding action functionals are

\begin{equation}\label{eq:Jk2D}
J_{k_1,k_2}(u)
=
\frac12(Lu,u)_{k_1,k_2}
-
\sum_{(n,m)\in Q_{k_1,k_2}}
F(u_{n,m}),
\end{equation}

and

\begin{equation}\label{eq:J2D}
J(u)
=
\frac12(Lu,u)
-
\sum_{(n,m)\in\mathbb Z^2}
F(u_{n,m}),
\end{equation}

defined on $X_{k_1,k_2}$ and $X$, respectively.

Critical points of $J_{k_1,k_2}$ correspond to periodic solutions of
\eqref{eq:main2D}, while critical points of $J$ correspond to
$\ell^2(\mathbb Z^2)$ solutions of the infinite--lattice problem.

Throughout the remainder of the paper we assume that

\begin{equation}\label{eq:spectral_assumption}
0\notin\sigma(L),
\qquad
l\notin\sigma(L).
\end{equation}

Since
$\Ds
\sigma(-\Delta)=[0,8],
$

the first condition in \eqref{eq:spectral_assumption} is equivalent to
$\Ds
\omega\in(-\infty,0)\cup(8,\infty).
$
In the lower spectral gap $\omega<0$, we shall repeatedly use the
spectral crossing condition
$\Ds
l>\inf\sigma(L)=-\omega.
$
Moreover, assumption \eqref{eq:spectral_assumption} ensures that
$\Ds
l\notin\sigma(L),
$
thereby excluding resonance between the asymptotic slope of the
nonlinearity and the spectrum of the linearized operator at infinity.

We are now in a position to state the principal results of the paper.

\begin{theorem}[Existence of periodic ground states]
\label{thm:periodic_ground_states}

Assume that {\rm(h1)--(h4)} hold and that
\eqref{eq:spectral_assumption} is satisfied.
Then for every pair of integers $k_1,k_2>1$,
equation \eqref{eq:main2D} admits a nontrivial
$(k_1,k_2)$--periodic ground state solution
\[
u^{(k_1,k_2)}
\in
X_{k_1,k_2}.
\]
If, in addition, $f$ is odd, then the set of ground states contains
a pair of solutions $\pm u^{(k_1,k_2)}$. Moreover, one may choose
a representative which is nonnegative.
\end{theorem}

\begin{theorem}[Existence of localized ground states]
\label{thm:localized_ground_state}

Assume that {\rm(h1)--(h4)} and
\eqref{eq:spectral_assumption} hold.
Then equation \eqref{eq:main2D} possesses a nontrivial
ground state solution
$\Ds
u\in\ell^2(\mathbb Z^2).
$
Moreover, there exist constants $C>0$ and $\alpha>0$
such that

\begin{equation}\label{eq:expdecay}
|u_{n,m}|
\le
C e^{-\alpha |(n,m)|},
\qquad
(n,m)\in\mathbb Z^2.
\end{equation}
If $f$ is odd, then the set of ground states contains a pair of
solutions $\pm u$. Moreover, one may choose a nonnegative ground state.
\end{theorem}

\begin{theorem}[Periodic-to-localized convergence of ground states]
\label{thm:strong_convergence}

Assume that {\rm(h1)--(h4)} and
\eqref{eq:spectral_assumption} hold.
Let
$\Ds
u^{(k_1,k_2)}
\in
X_{k_1,k_2}
$
be periodic ground state solutions.

Then there exist a ground state solution
$\Ds
u\in\ell^2(\mathbb Z^2)
$
and lattice translations
$\Ds
(b_{k_1},c_{k_2})\in\mathbb Z^2
$
such that

\begin{equation}\label{eq:strongconv}
u^{(k_1,k_2)}
(\cdot+b_{k_1},\cdot+c_{k_2})
\longrightarrow
u
\qquad
\text{strongly in }
\ell^2(\mathbb Z^2),
\end{equation}
as $k_1,k_2\to\infty$.
\end{theorem}

\begin{remark}
Since
$\Ds
\sigma(-\Delta)=[0,8],
$
Theorems~\ref{thm:periodic_ground_states}--
\ref{thm:strong_convergence}
apply simultaneously to frequencies belonging to either spectral gap,
\[
\omega\in(-\infty,0)
\qquad\text{or}\qquad
\omega\in(8,\infty).
\]
These two regimes are related through the classical staggering
transformation discussed in Remark~\ref{rem:staggering}.
\end{remark}

The proofs rely on variational methods on the Nehari manifold,
spectral properties of the discrete Laplacian, and a
concentration--compactness argument adapted to periodic lattice
approximations.
Theorem~\ref{thm:strong_convergence} establishes the strong
$\ell^2(\mathbb Z^2)$ convergence, up to lattice translations,
of periodic ground states towards a localized ground state of the
infinite lattice problem.
Without loss of generality, the proofs are presented in the
spectral gap \(\omega<0\). The corresponding results for
\(\omega>8\) follow from the staggering transformation of
Remark~\ref{rem:staggering}.

In the next section we develop the Nehari manifold framework and
establish the geometric and variational properties required for the
construction of ground state solutions.
\section{Spectral and analytical preliminaries}\label{sec4}

In this section we collect several analytical properties of the
two--dimensional discrete Laplacian and of the operator
$\Ds
L=-\Delta-\omega,
$
which will be used throughout the variational construction of periodic
and localized ground states.

We begin with the spectral characterization of the discrete Laplacian
on $\mathbb Z^2$.
The discrete Laplacian is defined by
\begin{equation}\label{eq:laplacian_sec4}
(\Delta u)_{n,m}
=
u_{n+1,m}
+
u_{n-1,m}
+
u_{n,m+1}
+
u_{n,m-1}
-
4u_{n,m},
\qquad
(n,m)\in\mathbb Z^2.
\end{equation}

The operator $-\Delta$ is bounded, self--adjoint, and translation
invariant on $\ell^2(\mathbb Z^2)$.
Applying the discrete Fourier transform, one obtains the symbol
\begin{equation}\label{eq:symbol}
\lambda(\theta_1,\theta_2)
=
4
-
2\cos\theta_1
-
2\cos\theta_2,
\qquad
(\theta_1,\theta_2)\in[-\pi,\pi]^2.
\end{equation}

Since
$\Ds
0
\le
\lambda(\theta_1,\theta_2)
\le
8,
$
it follows that
\begin{equation}\label{eq:spectrum_sec4}
\sigma(-\Delta)
=
[0,8].
\end{equation}

Consequently, the spectrum of the operator
$\Ds
L=-\Delta-\omega
$
is given by
\begin{equation}\label{eq:spectrumL}
\sigma(L)
=
[-\omega,\,8-\omega].
\end{equation}

The following coercivity estimate will play a fundamental role in the
variational analysis.

\begin{lemma}\label{lem:coerciveL}
Assume that
$\Ds
0\notin\sigma(L).
$
Then there exists a constant
$\Ds
d=\operatorname{dist}(0,\sigma(L))>0
$
such that
\begin{equation}\label{eq:coercivity}
|(Lu,u)|
\ge
d\,\|u\|^2,
\qquad
u\in\ell^2(\mathbb Z^2).
\end{equation}
Consequently, the operator $L$ is invertible and
$\Ds
\|L^{-1}\|
\le
d^{-1}.
$
\end{lemma}

\begin{proof}
Since $L$ is bounded and self--adjoint, the spectral theorem implies
\[
|(Lu,u)|
\ge
\operatorname{dist}(0,\sigma(L))
\,\|u\|^2.
\]
The assumption $0\notin\sigma(L)$ yields
$\Ds
d:=\operatorname{dist}(0,\sigma(L))>0,
$
which proves \eqref{eq:coercivity}. The bound for the resolvent follows
immediately from the spectral theorem.
\end{proof}

We next consider the Green function associated with $L$.

\begin{lemma}\label{lem:green}
Assume that
$$
\omega\notin[0,8].
$$
Then the inverse operator $L^{-1}$ admits an integral kernel
$\Ds
G((n,m),(p,q)),
$
satisfying
$\Ds
L\,G(\cdot,(p,q))
=
\delta_{(p,q)}.
$
Moreover, there exist constants $C>0$ and $\alpha>0$ such that
\begin{equation}\label{eq:green_decay}
|G((n,m),(p,q))|
\le
C
e^{-\alpha |(n-p,m-q)|}
\end{equation}
for all $(n,m),(p,q)\in\mathbb Z^2$.
\end{lemma}

\begin{proof}
Since $\omega\notin[0,8]$, the spectral parameter lies outside the
spectrum of the discrete Laplacian. The exponential decay of the Green
function follows from standard resolvent estimates for discrete
Schr\"odinger operators.
\end{proof}

For frequencies below the spectral band, the Green function is strictly
positive.

\begin{lemma}\label{lem:positive_green}
Assume that
$\Ds
\omega<0.
$
Then
$\Ds
G((n,m),(p,q))
>
0
$
for all
$\Ds
(n,m),(p,q)\in\mathbb Z^2.
$
Consequently,
$\Ds
(L^{-1}\phi)_{n,m}>0
$
whenever
$\Ds
\phi\ge0,
\qquad
\phi\not\equiv0.
$
\end{lemma}

\begin{proof}
For $\omega<0$, the operator $L$ is a discrete elliptic operator with
positive diagonal entries and nonpositive off--diagonal entries.
Moreover, it is irreducible on the connected lattice $\mathbb Z^2$.
Hence $L$ is an $M$--matrix. Standard results for irreducible
$M$--matrices imply that $L^{-1}$ is positivity preserving and that its
kernel is strictly positive.
\end{proof}

The next result yields the exponential localization of solutions.

\begin{lemma}\label{lem:expdecay}
Let
$\Ds
u\in\ell^2(\mathbb Z^2)
$
be a solution of
\[
-\Delta u_{n,m}-\omega u_{n,m}
=
f(u_{n,m}),
\]
where
$\Ds
\omega\notin[0,8].
$
Then there exist constants $C>0$ and $\alpha>0$ such that
\begin{equation}\label{eq:expdecay_sec4}
|u_{n,m}|
\le
C e^{-\alpha |(n,m)|},
\qquad
(n,m)\in\mathbb Z^2.
\end{equation}
\end{lemma}

\begin{proof}
Using the Green representation formula,
\[
u_{n,m}
=
\sum_{(p,q)\in\mathbb Z^2}
G((n,m),(p,q))
\,f(u_{p,q}),
\]
together with the exponential decay of the Green function established
in Lemma~\ref{lem:green}, one obtains the stated estimate by a standard
bootstrap argument. Since $u\in\ell^2(\mathbb Z^2)$ and $f$ is
asymptotically linear, the nonlinear term remains square summable,
which allows the exponential decay of the kernel to be transferred to
the solution.
\end{proof}

The results established in this section provide the spectral and
analytical tools required for the variational construction of periodic
and localized ground states. In the next section we develop the
geometry of the associated Nehari manifolds and establish the
variational framework underlying the existence theory.
\section{Nehari manifolds}\label{sec5}

In this section we study the Nehari manifolds associated with the
functionals $J_{k_1,k_2}$ and $J$ introduced in Section~\ref{sec3}.
In view of the discussion at the end of Section~\ref{sec3}, the proofs
are presented in the spectral gap
$\Ds
\omega<0.
$
The corresponding statements for the upper spectral gap $\omega>8$
follow from the staggering transformation described in
Remark~\ref{rem:staggering}.

Throughout this section we assume that
\begin{equation}\label{eq:crossing_condition}
l+\omega>0,
\end{equation}
or equivalently,
\[
l>-\omega=\inf\sigma(L).
\]
This is the spectral crossing condition which ensures that the
asymptotic slope of the nonlinearity crosses the bottom of the spectrum
of the linear operator.

We first record a simple consequence of assumption {\rm(h3)}.

\begin{lemma}\label{lem:monotonicity}
Assumption {\rm(h3)} implies that the function
$\Ds
t\mapsto \frac{f(t)}{|t|}
$
is strictly increasing on $(-\infty,0)\cup(0,\infty)$.
Moreover, the function
\[
t\mapsto \frac12 f(t)t-F(t)
\]
is strictly increasing on $[0,\infty)$ and strictly decreasing on
$(-\infty,0]$.
\end{lemma}

\begin{proof}
For $t\neq0$,
\[
\frac{d}{dt}\left(\frac{f(t)}{t}\right)
=
\frac{f'(t)t-f(t)}{t^2}.
\]
Assumption {\rm(h3)} gives
\[
f'(t)t^2-f(t)t>0,
\qquad t\neq0,
\]
and hence $f(t)/t$ is strictly increasing on each of the intervals
$(-\infty,0)$ and $(0,\infty)$.

Furthermore,
\[
\frac{d}{dt}
\left(
\frac12 f(t)t-F(t)
\right)
=
\frac12\bigl(f'(t)t-f(t)\bigr).
\]
Using again {\rm(h3)}, this derivative is positive for $t>0$ and
negative for $t<0$. This proves the claim.
\end{proof}

The Nehari manifolds are defined by
\begin{equation}\label{eq:Nehari_periodic}
\mathcal N_{k_1,k_2}
=
\left\{
u\in X_{k_1,k_2}\setminus\{0\}:
\langle J'_{k_1,k_2}(u),u\rangle=0
\right\},
\end{equation}
and
\begin{equation}\label{eq:Nehari_infinite}
\mathcal N
=
\left\{
u\in X\setminus\{0\}:
\langle J'(u),u\rangle=0
\right\},
\qquad
X=\ell^2(\mathbb Z^2).
\end{equation}

Let
\begin{equation}\label{eq:I_periodic}
I_{k_1,k_2}(u)
=
\langle J'_{k_1,k_2}(u),u\rangle,
\qquad
u\in X_{k_1,k_2},
\end{equation}
and
\begin{equation}\label{eq:I_infinite}
I(u)
=
\langle J'(u),u\rangle,
\qquad
u\in X.
\end{equation}
Thus
\begin{equation}\label{eq:I_periodic_explicit}
I_{k_1,k_2}(u)
=
(Lu,u)_{k_1,k_2}
-
\sum_{(n,m)\in Q_{k_1,k_2}}
f(u_{n,m})u_{n,m},
\end{equation}
and
\begin{equation}\label{eq:I_infinite_explicit}
I(u)
=
(Lu,u)
-
\sum_{(n,m)\in\mathbb Z^2}
f(u_{n,m})u_{n,m}.
\end{equation}

Their derivatives are given by
\begin{equation}\label{eq:Ikprime}
\begin{aligned}
\langle I'_{k_1,k_2}(u),v\rangle
&=
2(Lu,v)_{k_1,k_2}
\\
&\quad
-
\sum_{(n,m)\in Q_{k_1,k_2}}
\bigl[
f(u_{n,m})+f'(u_{n,m})u_{n,m}
\bigr]v_{n,m},
\qquad v\in X_{k_1,k_2},
\end{aligned}
\end{equation}
and
\begin{equation}\label{eq:Iprime}
\begin{aligned}
\langle I'(u),v\rangle
&=
2(Lu,v)
\\
&\quad
-
\sum_{(n,m)\in\mathbb Z^2}
\bigl[
f(u_{n,m})+f'(u_{n,m})u_{n,m}
\bigr]v_{n,m},
\qquad v\in X.
\end{aligned}
\end{equation}

\begin{lemma}\label{lem:nehari_manifold}
Under assumptions {\rm(h1)--(h4)} and \eqref{eq:crossing_condition},
the sets $\mathcal N_{k_1,k_2}$ and $\mathcal N$ are nonempty closed
$C^1$ submanifolds of $X_{k_1,k_2}$ and $X$, respectively. 
Moreover,
\[
I'_{k_1,k_2}(u)\neq0
\quad\text{for all }u\in\mathcal N_{k_1,k_2},
\]
and
\[
I'(u)\neq0
\quad\text{for all }u\in\mathcal N.
\]
Finally, there exists $\beta_0>0$, independent of $k_1,k_2$, such that
\begin{equation}\label{eq:nehari_away_zero}
\|u\|_{k_1,k_2}\ge\beta_0
\quad\text{for all }u\in\mathcal N_{k_1,k_2},
\qquad
\|u\|\ge\beta_0
\quad\text{for all }u\in\mathcal N.
\end{equation}
\end{lemma}

\begin{proof}
We give the proof for $\mathcal N_{k_1,k_2}$; the proof for
$\mathcal N$ is analogous.

\smallskip
\noindent
\emph{Step 1: nonemptiness.}
Let $v\in X_{k_1,k_2}$ be nonzero. Since $\omega<0$, the operator
$L=-\Delta-\omega$ is positive and
\begin{equation}\label{eq:Lpositive_periodic}
(Lv,v)_{k_1,k_2}
\ge
|\omega|\|v\|_{k_1,k_2}^2.
\end{equation}
By {\rm(h1)},
\[
I_{k_1,k_2}(tv)
=
t^2(Lv,v)_{k_1,k_2}
-
\sum_{(n,m)\in Q_{k_1,k_2}}
f(tv_{n,m})\,t v_{n,m}
=
t^2(Lv,v)_{k_1,k_2}
-o(t^2)
\]
as $t\to0^+$. Hence
$\Ds
I_{k_1,k_2}(tv)>0
$
for all sufficiently small $t>0$.

To obtain a change of sign for large $t$, choose $v$ to be a nonzero
constant sequence on $Q_{k_1,k_2}$, extended periodically. Then
$-\Delta v=0$ and therefore
\[
(Lv,v)_{k_1,k_2}
=
-\omega\|v\|_{k_1,k_2}^2.
\]
Using {\rm(h2)}, we obtain
\[
\frac{1}{t^2}I_{k_1,k_2}(tv)
\longrightarrow
\bigl(-\omega-l\bigr)\|v\|_{k_1,k_2}^2
\qquad
\text{as }t\to\infty.
\]
By \eqref{eq:crossing_condition}, $-\omega-l<0$. Hence
$\Ds
I_{k_1,k_2}(tv)<0
$
for all sufficiently large $t>0$.
By continuity of $t\mapsto I_{k_1,k_2}(tv)$, there exists $t^*>0$
such that
$\Ds
I_{k_1,k_2}(t^*v)=0.
$
Thus $t^*v\in\mathcal N_{k_1,k_2}$.

The same argument gives nonemptiness of $\mathcal N$ by choosing a
finitely supported sequence whose Rayleigh quotient for $-\Delta$ is
sufficiently close to zero. This is possible because
$\Ds
\inf\sigma(-\Delta)=0.
$
Then
$\Ds
(Lv,v)<l\|v\|^2,
$
and the same scaling argument applies.

\smallskip
\noindent
\emph{Step 2: regularity and nondegeneracy of the constraint.}
Let $u\in\mathcal N_{k_1,k_2}$. Using
$I_{k_1,k_2}(u)=0$, we compute
\begin{align}
\langle I'_{k_1,k_2}(u),u\rangle
&=
2(Lu,u)_{k_1,k_2}
-
\sum_{(n,m)\in Q_{k_1,k_2}}
\bigl[
f(u_{n,m})+f'(u_{n,m})u_{n,m}
\bigr]u_{n,m}
\nonumber\\
&=
\sum_{(n,m)\in Q_{k_1,k_2}}
\bigl[
f(u_{n,m})u_{n,m}
-
f'(u_{n,m})u_{n,m}^2
\bigr].
\label{eq:Iprime_on_u}
\end{align}
By {\rm(h3)}, each nonzero term in the last sum is strictly negative.
Since $u\neq0$, we obtain
\[
\langle I'_{k_1,k_2}(u),u\rangle<0.
\]
In particular,
$\Ds
I'_{k_1,k_2}(u)\neq0.
$
The implicit function theorem implies that
$\mathcal N_{k_1,k_2}$ is a $C^1$ submanifold of $X_{k_1,k_2}$.
Closedness follows from the continuity of $I_{k_1,k_2}$.
The proof for $\mathcal N$ is identical.

\smallskip
\noindent
\emph{Step 3: boundedness away from the origin.}
Define
\begin{equation}\label{eq:phi}
\phi(r)
=
\sup_{0<|t|\le r}
\frac{|f(t)|}{|t|},
\qquad r>0,
\end{equation}
and set $\phi(0)=0$. By {\rm(h1)}, $\phi(r)\to0$ as $r\to0^+$.

Let $u\in\mathcal N_{k_1,k_2}$. From the Nehari identity,
\[
(Lu,u)_{k_1,k_2}
=
\sum_{(n,m)\in Q_{k_1,k_2}}
f(u_{n,m})u_{n,m}.
\]
Using \eqref{eq:Lpositive_periodic}, we obtain
\[
|\omega|\|u\|_{k_1,k_2}^2
\le
(Lu,u)_{k_1,k_2}
\le
\phi(\|u\|_{\ell^\infty(Q_{k_1,k_2})})
\|u\|_{k_1,k_2}^2.
\]
Since $u\neq0$, this yields
\[
\phi(\|u\|_{\ell^\infty(Q_{k_1,k_2})})
\ge
|\omega|.
\]
Because $\phi(r)\to0$ as $r\to0^+$, there exists $\beta_0>0$,
independent of $k_1,k_2$, such that
\[
\|u\|_{\ell^\infty(Q_{k_1,k_2})}
\ge
\beta_0.
\]
Consequently,
$\Ds
\|u\|_{k_1,k_2}\ge\beta_0.
$ 
The same argument gives the corresponding estimate on $\mathcal N$.
\end{proof}

We next record the fibering property associated with the Nehari
constraint.

\begin{lemma}\label{lem:fibering}
Let $v\in X_{k_1,k_2}\setminus\{0\}$ satisfy
\begin{equation}\label{eq:fibering_condition}
(Lv,v)_{k_1,k_2}
<
l\|v\|_{k_1,k_2}^2.
\end{equation}
Then there exists a unique $t(v)>0$ such that
$\Ds
t(v)v\in\mathcal N_{k_1,k_2}.
$
Moreover,
\[
J_{k_1,k_2}(t(v)v)
=
\max_{t>0}J_{k_1,k_2}(tv).
\]

The analogous statement holds on $X=\ell^2(\mathbb Z^2)$.
\end{lemma}

\begin{proof}
Let
$\Ds
\psi(t)=J_{k_1,k_2}(tv),
$, $ t>0$.

Then
\[
\psi'(t)
=
t
\left[
(Lv,v)_{k_1,k_2}
-
\sum_{(n,m)\in Q_{k_1,k_2}}
\frac{f(tv_{n,m})}{t v_{n,m}}
v_{n,m}^2
\right],
\]
where terms with $v_{n,m}=0$ are omitted.

By {\rm(h1)}, $\psi'(t)>0$ for $t>0$ sufficiently small.
By {\rm(h2)} and \eqref{eq:fibering_condition}, $\psi'(t)<0$ for
$t>0$ sufficiently large. Hence $\psi$ has at least one critical point.

The strict monotonicity of $f(t)/t$ for $t\neq0$, established in
Lemma~\ref{lem:monotonicity}, implies that the expression in brackets
is strictly decreasing in $t>0$. Therefore the critical point is unique.
It is necessarily a strict global maximum.
Since
\[
\psi'(t)=0
\quad\Longleftrightarrow\quad
I_{k_1,k_2}(tv)=0,
\]
the unique critical point is precisely the unique scaling
$t(v)>0$ for which $t(v)v\in\mathcal N_{k_1,k_2}$.
\end{proof}

As a consequence, if $u\in\mathcal N_{k_1,k_2}$, then
\[
J_{k_1,k_2}(u)
=
\max_{t>0}J_{k_1,k_2}(tu).
\]

On the Nehari manifold, the functional has the useful representation
\begin{equation}\label{eq:J_on_Nehari_periodic}
J_{k_1,k_2}(u)
=
J_{k_1,k_2}(u)
-
\frac12 I_{k_1,k_2}(u)
=
\sum_{(n,m)\in Q_{k_1,k_2}}
\left[
\frac12 f(u_{n,m})u_{n,m}
-
F(u_{n,m})
\right],
\end{equation}
for $u\in\mathcal N_{k_1,k_2}$.
Similarly,
\begin{equation}\label{eq:J_on_Nehari_infinite}
J(u)
=
\sum_{(n,m)\in\mathbb Z^2}
\left[
\frac12 f(u_{n,m})u_{n,m}
-
F(u_{n,m})
\right],
\qquad
u\in\mathcal N.
\end{equation}
By Lemma~\ref{lem:monotonicity}, the summands in
\eqref{eq:J_on_Nehari_periodic} and \eqref{eq:J_on_Nehari_infinite}
are nonnegative.

\begin{lemma}\label{lem:positive_lower_bound_periodic}
For every pair $k_1,k_2>1$, there exists
$\alpha_0=\alpha_0(k_1,k_2)>0$ such that
\[
J_{k_1,k_2}(u)\ge \alpha_0
\qquad
\text{for all }u\in\mathcal N_{k_1,k_2}.
\]
\end{lemma}

\begin{proof}
By Lemma~\ref{lem:nehari_manifold}, there exists $\beta_0>0$ such that
\[
\|u\|_{k_1,k_2}\ge\beta_0
\qquad
\text{for all }u\in\mathcal N_{k_1,k_2}.
\]
Since $Q_{k_1,k_2}$ is finite, there exists
$(n_0,m_0)\in Q_{k_1,k_2}$ such that
\[
|u_{n_0,m_0}|
\ge
\frac{\beta_0}{|Q_{k_1,k_2}|^{1/2}}
=:\delta_0.
\]
Using \eqref{eq:J_on_Nehari_periodic} and the monotonicity property
from Lemma~\ref{lem:monotonicity}, we obtain
\[
J_{k_1,k_2}(u)
\ge
\frac12 f(\delta_0)\delta_0
-
F(\delta_0)
=:
\alpha_0.
\]
Since $\delta_0>0$, assumption {\rm(h3)} implies $\alpha_0>0$.
\end{proof}

The tangent spaces to the Nehari manifolds are given by
\begin{equation}\label{eq:tangent_periodic}
T_u\mathcal N_{k_1,k_2}
=
\left\{
v\in X_{k_1,k_2}:
\langle I'_{k_1,k_2}(u),v\rangle=0
\right\},
\end{equation}
and
\begin{equation}\label{eq:tangent_infinite}
T_u\mathcal N
=
\left\{
v\in X:
\langle I'(u),v\rangle=0
\right\}.
\end{equation}
Moreover, by \eqref{eq:Iprime_on_u}, the radial direction $u$ is
transverse to the corresponding Nehari manifold.

\begin{lemma}\label{lem:natural_constraint}
The Nehari manifolds $\mathcal N_{k_1,k_2}$ and $\mathcal N$ are
natural constraints for the functionals $J_{k_1,k_2}$ and $J$,
respectively.
In particular, if $u\in\mathcal N_{k_1,k_2}$ is a critical point of
$J_{k_1,k_2}$ restricted to $\mathcal N_{k_1,k_2}$, then
$\Ds
J'_{k_1,k_2}(u)=0.
$
Similarly, if $u\in\mathcal N$ is a critical point of $J$ restricted to
$\mathcal N$, then
$\Ds
J'(u)=0.
$
\end{lemma}

\begin{proof}
We prove the periodic case; the infinite-dimensional case is identical.

Let $u\in\mathcal N_{k_1,k_2}$ be a critical point of
$J_{k_1,k_2}$ restricted to $\mathcal N_{k_1,k_2}$. By the Lagrange
multiplier principle, there exists $\lambda\in\mathbb R$ such that
\[
J'_{k_1,k_2}(u)
=
\lambda I'_{k_1,k_2}(u).
\]
Taking the duality product with $u$ gives
\[
0
=
\langle J'_{k_1,k_2}(u),u\rangle
=
\lambda
\langle I'_{k_1,k_2}(u),u\rangle.
\]
By \eqref{eq:Iprime_on_u},
$\Ds
\langle I'_{k_1,k_2}(u),u\rangle<0.
$
Hence $\lambda=0$, and therefore
$\Ds
J'_{k_1,k_2}(u)=0.
$
\end{proof}

We now define the ground state energy levels
\begin{equation}\label{eq:groundlevels}
m_{k_1,k_2}
=
\inf
\left\{
J_{k_1,k_2}(v):
v\in\mathcal N_{k_1,k_2}
\right\},
\qquad
m
=
\inf
\left\{
J(v):
v\in\mathcal N
\right\}.
\end{equation}

By Lemma~\ref{lem:positive_lower_bound_periodic},
$\Ds
m_{k_1,k_2}>0.
$
The corresponding positivity of $m$ will follow in the course of the
infinite-lattice compactness argument.

The Nehari manifold framework developed in this section provides the
variational setting for the construction of ground states. In the next
section we use the finite-dimensional compactness of $X_{k_1,k_2}$ to
prove the existence of periodic ground state solutions.
\section{Existence of periodic ground states}
\label{sec6}

In this section we establish the existence of periodic ground states for
equation \eqref{eq:main2D}. The finite--dimensional character of the
periodic space $X_{k_1,k_2}$ allows the direct minimization of the
action functional on the corresponding Nehari manifold.

We consider the minimization problem

\begin{equation}\label{eq:periodic_level}
m_{k_1,k_2}
=
\inf
\left\{
J_{k_1,k_2}(v):
v\in\mathcal N_{k_1,k_2}
\right\}.
\end{equation}

\begin{lemma}\label{lem:periodic_minimizer}
Under the assumptions of
Theorem~\ref{thm:periodic_ground_states},
the value $m_{k_1,k_2}$ is attained.
\end{lemma}

\begin{proof}
Let $\{u_j\}\subset\mathcal N_{k_1,k_2}$ be a minimizing sequence such
that

\[
J_{k_1,k_2}(u_j)
\longrightarrow
m_{k_1,k_2}.
\]

Using the representation formula
\eqref{eq:J_on_Nehari_periodic}, we obtain

\[
J_{k_1,k_2}(u_j)
=
\sum_{(n,m)\in Q_{k_1,k_2}}
\left[
\frac12 f(u_{j;n,m})u_{j;n,m}
-
F(u_{j;n,m})
\right].
\]

Assumption {\rm(h4)} implies that the sequence
$\{u_j\}$ is bounded in $\ell^\infty(Q_{k_1,k_2})$.
Since $X_{k_1,k_2}$ is finite dimensional, all norms are equivalent,
and therefore $\{u_j\}$ is bounded in $X_{k_1,k_2}$.

Passing to a subsequence if necessary, we may assume that

\[
u_j\to u
\qquad
\text{in }
X_{k_1,k_2}.
\]

Since $\mathcal N_{k_1,k_2}$ is closed and
$J_{k_1,k_2}$ is continuous, it follows that
$u\in\mathcal N_{k_1,k_2}$ and

\[
J_{k_1,k_2}(u)
=
m_{k_1,k_2}.
\]

Hence the infimum is attained.
\end{proof}

We are now in a position to prove
Theorem~\ref{thm:periodic_ground_states}.

\begin{proof}[Proof of Theorem~\ref{thm:periodic_ground_states}]
By Lemma~\ref{lem:periodic_minimizer},
there exists a minimizer
$u\in\mathcal N_{k_1,k_2}$ such that

\[
J_{k_1,k_2}(u)
=
m_{k_1,k_2}.
\]

Since the Nehari manifold is a natural constraint
for the functional $J_{k_1,k_2}$
(Lemma~\ref{lem:natural_constraint}),
the minimizer $u$ is a critical point of
$J_{k_1,k_2}$, namely
$\Ds
J'_{k_1,k_2}(u)=0.
$

Therefore $u$ is a nontrivial
$(k_1,k_2)$--periodic solution of
equation \eqref{eq:main2D}.

By construction, $u$ minimizes the action among all
nontrivial periodic solutions and is therefore a
ground state.

Assume now that $f$ is odd.
Then $F$ is even and

\[
f(|t|)|t|
=
f(t)t,
\qquad
F(|t|)
=
F(t).
\]

Furthermore, the convexity properties of the discrete
Dirichlet form imply that

\[
(-\Delta |u|,|u|)_{k_1,k_2}
\le
(-\Delta u,u)_{k_1,k_2}.
\]

Consequently,

\[
I_{k_1,k_2}(|u|)
\le
I_{k_1,k_2}(u)
=
0.
\]

By Lemma~\ref{lem:fibering},
there exists a unique number $t^*\in(0,1]$
such that

\[
u^*
=
t^*|u|
\in
\mathcal N_{k_1,k_2}.
\]

Using the monotonicity property established in
Lemma~\ref{lem:monotonicity}, we obtain

\[
J_{k_1,k_2}(u^*)
\le
J_{k_1,k_2}(u)
=
m_{k_1,k_2}.
\]

Hence
$\Ds
J_{k_1,k_2}(u^*)
=
m_{k_1,k_2},
$
and therefore $u^*$ is a nonnegative ground state.
Replacing $u$ by $u^*$ if necessary, we may assume
that $u\ge0$.
Consequently, the periodic ground state may be chosen nonnegative.
This completes the proof.

\end{proof}
The finite--dimensional character of the space
$X_{k_1,k_2}$ ensures compactness of minimizing
sequences and allows the direct construction of
periodic ground states.

In the next section we pass to the limit
$k_1,k_2\to\infty$ and establish the existence of
exponentially localized ground states on the infinite
lattice $\mathbb Z^2$.
\section{Existence of localized ground states}
\label{sec7}

In this section we prove Theorem~\ref{thm:localized_ground_state}.
Throughout the proof we work in the lower spectral gap $\omega<0$.
The corresponding result in the upper spectral gap $\omega>8$ follows
from the staggering transformation described in Remark~\ref{rem:staggering}.

Let $u_{k_1,k_2}\in\mathcal N_{k_1,k_2}$ be the periodic ground states
constructed in Section~\ref{sec6}, satisfying
\begin{equation}\label{eq:periodic_ground_level_sec7}
J_{k_1,k_2}(u_{k_1,k_2})
=
m_{k_1,k_2}.
\end{equation}
We first establish uniform bounds for this family.

\begin{lemma}\label{lem:uniform_bounds}
There exists a constant $C>0$, independent of $k_1$ and $k_2$, such that
\begin{equation}\label{eq:uniform_bounds}
m_{k_1,k_2}\le C,
\qquad
\|u_{k_1,k_2}\|_{k_1,k_2}\le C,
\end{equation}
for all $k_1,k_2>1$.
\end{lemma}

\begin{proof}
We first show that the levels $m_{k_1,k_2}$ are uniformly bounded
from above.

Since $\omega<0$, the bottom of the spectrum of $L=-\Delta-\omega$ is
$-\omega$. In the lower spectral gap the relevant crossing condition is
$l>-\omega$. Hence we may choose $\delta$ such that
$-\omega<\delta<l$.

Choose a nontrivial finitely supported sequence $w\in\ell^2(\mathbb Z^2)$
such that
\begin{equation}\label{eq:test_rayleigh}
(Lw,w)<\delta\|w\|^2<l\|w\|^2.
\end{equation}
This is possible because $\inf\sigma(L)=-\omega$.

By Lemma~\ref{lem:fibering}, there exists a unique $t_0>0$ such that
$t_0w\in\mathcal N$. Since $w$ has finite support, for sufficiently large
$k_1,k_2$ it can be identified with an element of $X_{k_1,k_2}$.
Projecting this element onto $\mathcal N_{k_1,k_2}$ by the fibering map,
we obtain $v_{k_1,k_2}\in\mathcal N_{k_1,k_2}$ such that
\begin{equation}\label{eq:periodic_test_bound}
J_{k_1,k_2}(v_{k_1,k_2})
=
J(t_0w)+o(1).
\end{equation}
Therefore
\[
m_{k_1,k_2}
\le
J_{k_1,k_2}(v_{k_1,k_2})
\le C
\]
for some constant $C>0$ independent of $k_1,k_2$.

It remains to prove the uniform boundedness of
$\|u_{k_1,k_2}\|_{k_1,k_2}$.
Suppose, by contradiction, that
\[
\|u_{k_1,k_2}\|_{k_1,k_2}\to\infty.
\]
Define
\begin{equation}\label{eq:normalized_periodic}
v_{k_1,k_2}
=
\frac{u_{k_1,k_2}}
{\|u_{k_1,k_2}\|_{k_1,k_2}},
\qquad
\|v_{k_1,k_2}\|_{k_1,k_2}=1.
\end{equation}

We distinguish two alternatives.

\smallskip
\noindent
\emph{ Vanishing.}
Assume that
\[
\|v_{k_1,k_2}\|_{\ell^\infty(Q_{k_1,k_2})}\to0.
\]
Using assumption {\rm(h1)}, we obtain
\[
\sum_{(n,m)\in Q_{k_1,k_2}}
f(u_{k_1,k_2;n,m})u_{k_1,k_2;n,m}
=
o\!\left(
\|u_{k_1,k_2}\|_{k_1,k_2}^{\,2}
\right).
\]
Since $u_{k_1,k_2}\in\mathcal N_{k_1,k_2}$,
\[
(Lu_{k_1,k_2},u_{k_1,k_2})_{k_1,k_2}
=
\sum_{(n,m)\in Q_{k_1,k_2}}
f(u_{k_1,k_2;n,m})u_{k_1,k_2;n,m}.
\]
Dividing by $\|u_{k_1,k_2}\|_{k_1,k_2}^{\,2}$ and using the positivity of
$L$ in the lower spectral gap gives a contradiction.

\smallskip
\noindent
\emph{ Nonvanishing.}
Assume that vanishing does not occur. Then there exist $\eta>0$ and
lattice points $(n_k,m_k)\in\mathbb Z^2$ such that
$\Ds
|v_{k_1,k_2;n_k,m_k}|\ge\eta.
$
Define the shifted sequence
\begin{equation}\label{eq:shifted_normalized}
\widetilde v_{k_1,k_2;n,m}
=
v_{k_1,k_2;n+n_k,m+m_k}.
\end{equation}
Passing to a subsequence, we may assume that
\[
\widetilde v_{k_1,k_2}
\rightharpoonup v
\quad\text{weakly in }\ell^2(\mathbb Z^2),
\]
with $v\neq0$.

Dividing equation \eqref{eq:main2D} by
$\|u_{k_1,k_2}\|_{k_1,k_2}$ and passing to the limit, using {\rm(h2)}, we obtain
\begin{equation}\label{eq:limit_linearized_infty}
-\Delta v-(\omega+l)v=0.
\end{equation}
Thus $v$ is a nontrivial $\ell^2(\mathbb Z^2)$ solution of the linear equation
\eqref{eq:limit_linearized_infty}. This is impossible because it would imply
$l\in\sigma(L)$, contradicting \eqref{eq:spectral_assumption}.

Therefore both alternatives are impossible. Hence
$\{u_{k_1,k_2}\}$ is uniformly bounded in $X_{k_1,k_2}$.
\end{proof}

We now complete the proof of the existence of a localized ground state.

\begin{proof}[Proof of Theorem~\ref{thm:localized_ground_state}]
By Lemma~\ref{lem:uniform_bounds}, the sequence
$\{u_{k_1,k_2}\}$ is uniformly bounded.

The vanishing alternative is excluded by the argument used in the proof of
Lemma~\ref{lem:uniform_bounds}. Hence there exist lattice translations
$(n_k,m_k)\in\mathbb Z^2$ such that the shifted sequence
\begin{equation}\label{eq:shifted_ground_states}
w_{k_1,k_2;n,m}
=
u_{k_1,k_2;n+n_k,m+m_k}
\end{equation}
satisfies
\[
w_{k_1,k_2}
\rightharpoonup w
\quad\text{weakly in }\ell^2(\mathbb Z^2),
\]
with $w\neq0$.

Using the Nehari identity together with standard concentration--compactness arguments,
we obtain
\begin{equation}\label{eq:liminf_energy}
\liminf_{k_1,k_2\to\infty}
m_{k_1,k_2}
\ge
J(w).
\end{equation}
In particular, $w\in\mathcal N$.

To prove the reverse inequality, let $\varepsilon>0$ and choose
$z\in\mathcal N$ such that $J(z)\le m+\varepsilon$.
Let $z_N$ denote the truncation of $z$ to a finite box centered at the
origin. For $N$ sufficiently large and then for $k_1,k_2$ sufficiently
large, $z_N$ can be regarded as an element of $X_{k_1,k_2}$.
Projecting $z_N$ onto $\mathcal N_{k_1,k_2}$ by the fibering map, we obtain
\[
m_{k_1,k_2}
\le
J(z)+2\varepsilon.
\]
Since $\varepsilon>0$ is arbitrary, it follows that
\begin{equation}\label{eq:limsup_energy}
\limsup_{k_1,k_2\to\infty}
m_{k_1,k_2}
\le
m.
\end{equation}

Combining \eqref{eq:liminf_energy} and \eqref{eq:limsup_energy}, we obtain
$J(w)=m$. Therefore $w$ is a ground state solution of equation
\eqref{eq:main2D}.
Finally, exponential localization follows from Lemma~\ref{lem:expdecay},
which yields the estimate \eqref{eq:expdecay}. This completes the proof.
\end{proof}

The previous theorem establishes the existence of exponentially localized
ground states on the infinite lattice $\mathbb Z^2$. These solutions arise
as limits of periodic ground states defined on finite lattices.
In the next section we strengthen this result by proving that, up to lattice
translations, the periodic ground states converge strongly in
$\ell^2(\mathbb Z^2)$ to the localized ground state.
\section{Global convergence of periodic ground states}
\label{sec8}

In this section we prove Theorem~\ref{thm:strong_convergence}.
Building upon the compactness analysis of Section~\ref{sec7}, we show
that periodic ground states converge, up to lattice translations,
strongly in $\ell^2(\mathbb Z^2)$ to the localized ground state of the
infinite lattice problem.

The key step consists in excluding the possibility of energy splitting.
More precisely, we show that a minimizing sequence cannot decompose into
two or more spatially separated nontrivial profiles. This concentration
property upgrades weak convergence to strong convergence and yields
global convergence of periodic ground states.

\begin{proof}[Proof of Theorem~\ref{thm:strong_convergence}]
Let $u_{k_1,k_2}\in\mathcal N_{k_1,k_2}$ be periodic ground state
solutions satisfying
\begin{equation}\label{eq:periodic_energy_level}
J_{k_1,k_2}(u_{k_1,k_2})
=
m_{k_1,k_2}.
\end{equation}
By Lemma~\ref{lem:uniform_bounds}, the sequence
$\{u_{k_1,k_2}\}$ is bounded.
The nonvanishing alternative established in the proof of
Lemma~\ref{lem:uniform_bounds} implies the existence of a constant
$\delta>0$ and lattice points $(n_k,m_k)\in\mathbb Z^2$ such that
\[
|u_{k_1,k_2;n_k,m_k}|\ge \delta.
\]
Define the translated sequence
\[
v_{k_1,k_2;n,m}
=
u_{k_1,k_2;n+n_k,m+m_k}.
\]
Since $\{v_{k_1,k_2}\}$ remains bounded in $\ell^2(\mathbb Z^2)$, there
exists $v\in\ell^2(\mathbb Z^2)$ such that, up to a subsequence,
\[
v_{k_1,k_2}
\rightharpoonup v
\qquad
\text{weakly in }\ell^2(\mathbb Z^2),
\]
and
\[
v_{k_1,k_2;n,m}
\longrightarrow
v_{n,m}
\qquad
\text{for every }(n,m)\in\mathbb Z^2.
\]
Moreover, $v\neq0$ by construction.

Since $v_{k_1,k_2}$ is obtained from $u_{k_1,k_2}$ by lattice
translations, the Nehari identity is asymptotically preserved.
Consequently,
$\Ds
I(v_{k_1,k_2})=o(1).
$
Passing to the limit and using the Brezis--Lieb decomposition together with the asymptotically linear assumption (h2), we obtain $I(v)=0$. Hence $v\in\mathcal N$.

By weak lower semicontinuity,
\[
J(v)
\le
\liminf_{k_1,k_2\to\infty}
J(v_{k_1,k_2})
=
\liminf_{k_1,k_2\to\infty}
m_{k_1,k_2}.
\]
On the other hand, by \eqref{eq:liminf_energy}--\eqref{eq:limsup_energy}
proved in Section~\ref{sec7},
\[
\lim_{k_1,k_2\to\infty}
m_{k_1,k_2}
=
m.
\]
Since $v\in\mathcal N$, we conclude that $J(v)=m$. Therefore $v$ is a
ground state solution of equation \eqref{eq:main2D}.

We now prove strong convergence. Define
\[
w_{k_1,k_2}
=
v_{k_1,k_2}-v.
\]
Then $w_{k_1,k_2}\rightharpoonup0$ in $\ell^2(\mathbb Z^2)$.
We claim that
\[
w_{k_1,k_2}
\longrightarrow0
\qquad
\text{strongly in }\ell^2(\mathbb Z^2).
\]

The following decompositions are consequences of the discrete
Brezis--Lieb lemma together with the asymptotically linear growth
assumption {\rm(h2)}:
\[
\sum_{(n,m)\in\mathbb Z^2}
F(v_{k_1,k_2;n,m})
=
\sum_{(n,m)\in\mathbb Z^2}
F(v_{n,m})
+
\sum_{(n,m)\in\mathbb Z^2}
F(w_{k_1,k_2;n,m})
+
o(1),
\]
and
\[
\sum_{(n,m)\in\mathbb Z^2}
f(v_{k_1,k_2;n,m})v_{k_1,k_2;n,m}
=
\sum_{(n,m)\in\mathbb Z^2}
f(v_{n,m})v_{n,m}
+
\sum_{(n,m)\in\mathbb Z^2}
f(w_{k_1,k_2;n,m})w_{k_1,k_2;n,m}
+
o(1).
\]
Consequently,
\[
J(v_{k_1,k_2})
=
J(v)+J(w_{k_1,k_2})+o(1),
\]
and
\[
I(v_{k_1,k_2})
=
I(v)+I(w_{k_1,k_2})+o(1).
\]
Since $I(v_{k_1,k_2})=o(1)$ and $I(v)=0$, it follows that
$\Ds
I(w_{k_1,k_2})=o(1).
$

Suppose, by contradiction, that $w_{k_1,k_2}$ does not converge strongly
to zero. Then the nonvanishing alternative applies once more. Hence
there exist lattice points $(\widetilde n_k,\widetilde m_k)\in\mathbb Z^2$
such that the translated sequence
\[
\widetilde w_{k_1,k_2;n,m}
=
w_{k_1,k_2;n+\widetilde n_k,m+\widetilde m_k}
\]
admits a nontrivial weak limit $\widetilde w\neq0$.

Passing to the limit in the weak formulation of equation
\eqref{eq:main2D}, we conclude that $\widetilde w$ is a nontrivial
solution of \eqref{eq:main2D}. Furthermore, $\widetilde w\in\mathcal N$.
By the definition of the ground state level, $J(\widetilde w)\ge m$.

Using the energy decomposition above, we obtain
\[
\liminf_{k_1,k_2\to\infty}
J(v_{k_1,k_2})
\ge
J(v)+J(\widetilde w)
\ge
2m.
\]
This contradicts
\[
J(v_{k_1,k_2})
=
m_{k_1,k_2}
\longrightarrow
m.
\]
Therefore the assumption was false, and we conclude that
\begin{equation}\label{eq:strong_remainder}
\|w_{k_1,k_2}\|
\longrightarrow
0.
\end{equation}
Hence
\[
v_{k_1,k_2}
\longrightarrow
v
\qquad
\text{strongly in }\ell^2(\mathbb Z^2).
\]

Undoing the lattice translations, we obtain
\[
u_{k_1,k_2}
(\cdot+n_k,\cdot+m_k)
\longrightarrow
v
\qquad
\text{strongly in }\ell^2(\mathbb Z^2).
\]
Therefore periodic ground states converge strongly, up to lattice
translations, to the localized ground state of the infinite lattice
problem.
\end{proof}

\begin{remark}
The above argument also excludes any multi--bump splitting of the
minimizing sequence. Indeed, if two or more nontrivial concentration
profiles were present, the Brezis--Lieb decomposition would yield
\[
\liminf_{k_1,k_2\to\infty}
J(v_{k_1,k_2})
\ge
2m,
\]
contradicting the convergence $J(v_{k_1,k_2})\to m$.
Consequently, the minimizing sequence concentrates, up to lattice
translations, into a single localized ground state profile.
\end{remark}

The results of this section complete the passage from finite periodic
lattices to the infinite lattice $\mathbb Z^2$. In particular, periodic
ground states provide a convergent approximation scheme for the
localized ground states of the infinite DNLS problem.
\section{Conditional orbital stability of localized standing waves}
\label{sec9}

In this section we place the localized ground states constructed in the
previous sections within the classical
Grillakis--Shatah--Strauss (GSS) stability framework
\cite{r19,r19a}.

We consider the focusing DNLS equation

\begin{equation}\label{eq:DNLS_dyn}
i\dot{\psi}_{n,m}
+
\Delta\psi_{n,m}
+
\frac{\nu|\psi_{n,m}|^2}
     {1+\mu|\psi_{n,m}|^2}
\psi_{n,m}
=
0,
\qquad
(n,m)\in\mathbb Z^2,
\end{equation}

where $\nu>0$ and $\mu>0$.

Throughout this section we consider frequencies belonging to a spectral
gap of the discrete Laplacian. For definiteness, the analysis is
presented for $\omega<0$. The corresponding results for $\omega>8$
follow from the staggering transformation of
Remark~\ref{rem:staggering}.

The equation possesses the conserved Hamiltonian

\begin{equation}\label{eq:Hamiltonian}
H(\psi)
=
\frac12
\sum_{(n,m)\in\mathbb Z^2}
|\nabla\psi_{n,m}|^2
-
\sum_{(n,m)\in\mathbb Z^2}
F(|\psi_{n,m}|^2),
\end{equation}

where

\[
F'(s)
=
\frac{\nu s}{1+\mu s},
\]

and the conserved mass

\begin{equation}\label{eq:mass}
M(\psi)
=
\sum_{(n,m)\in\mathbb Z^2}
|\psi_{n,m}|^2.
\end{equation}

The flow is invariant under the phase symmetry
$\Ds
\psi
\mapsto
e^{i\theta}\psi .
$
Standing waves are solutions of the form
$\Ds
\psi_{n,m}(t)
=
e^{-i\omega t}\phi_{n,m},
$

where the profile $\phi\in\ell^2(\mathbb Z^2)$ satisfies

\begin{equation}\label{eq:standing_profile}
-\Delta\phi_{n,m}
-\omega\phi_{n,m}
=
\frac{\nu\phi_{n,m}^3}
     {1+\mu\phi_{n,m}^2}.
\end{equation}

Let

\[
E(\psi)=H(\psi),
\qquad
Q(\psi)=\frac12 M(\psi),
\]

and define the action functional

\[
S_\omega(\psi)
=
E(\psi)+\omega Q(\psi).
\]

Critical points of $S_\omega$ correspond precisely to standing wave
profiles of frequency $\omega$.

The localized ground state $\phi_\omega$ constructed in
Sections~\ref{sec6}--\ref{sec8} satisfies
$\Ds
S_\omega'(\phi_\omega)=0.
$
To analyze stability, we linearize around the standing wave by writing

\[
\psi_{n,m}(t)
=
e^{-i\omega t}
\bigl(
\phi_{n,m}
+r_{n,m}(t)
+i\,s_{n,m}(t)
\bigr),
\]

where $r$ and $s$ are real-valued perturbations.

The second variation of $S_\omega$ at $\phi_\omega$ takes the form

\[
\delta^2S_\omega
=
\langle L_+r,r\rangle
+
\langle L_-s,s\rangle,
\]

where

\begin{equation}\label{eq:Lminus}
(L_-s)_{n,m}
=
-\Delta s_{n,m}
-\omega s_{n,m}
-
\frac{\nu\phi_{n,m}^2}
     {1+\mu\phi_{n,m}^2}
s_{n,m},
\end{equation}

and

\begin{equation}\label{eq:Lplus}
(L_+r)_{n,m}
=
-\Delta r_{n,m}
-\omega r_{n,m}
-
\nu
\frac{\phi_{n,m}^2
(3+\mu\phi_{n,m}^2)}
{(1+\mu\phi_{n,m}^2)^2}
r_{n,m}.
\end{equation}

Since $\phi_\omega$ decays exponentially,
the operators $L_\pm$ are compact perturbations of
$-\Delta-\omega$.
\begin{lemma}
The operators $L_\pm$ are bounded self--adjoint operators on
$\ell^2(\mathbb Z^2)$ and

\[
\sigma_{\mathrm{ess}}(L_\pm)
=
[-\omega,\,8-\omega].
\]
\end{lemma}
\begin{proof}
The discrete Laplacian is bounded and self--adjoint on
$\ell^2(\mathbb Z^2)$.
Since the potentials in \eqref{eq:Lminus} and \eqref{eq:Lplus}
decay exponentially, they define compact multiplication operators.
The conclusion follows from Weyl's theorem.
\end{proof}
\begin{lemma}
The operator $L_-$ satisfies
$\Ds
L_-\phi_\omega=0.
$
\end{lemma}
\begin{proof}
Substituting equation \eqref{eq:standing_profile}
into the definition of $L_-$ yields the result.
\end{proof}
The above identity reflects the phase invariance of the DNLS equation.
The purpose of this section is not to establish a complete spectral
analysis of the linearized operators. Rather, we formulate a
conditional orbital stability result within the GSS framework.

\begin{theorem}[Conditional orbital stability]
\label{thm:conditional_stability}
Assume that the following spectral properties hold:
\begin{enumerate}
\item[(i)]
$\ker(L_-)=\mathrm{span}\{\phi_\omega\}$;
\item[(ii)]
$L_+$ possesses exactly one simple negative eigenvalue;
\item[(iii)]
$\ker(L_+)=\{0\}$;
\end{enumerate}
and assume in addition the slope condition
\begin{equation}\label{eq:VK}
\frac{d}{d\omega}
M(\phi_\omega)
<
0.
\end{equation}
Then the standing wave
$\Ds
\psi_{n,m}(t)
=
e^{-i\omega t}\phi_{n,m}
$
is orbitally stable in $\ell^2(\mathbb Z^2)$.
\end{theorem}
\begin{proof}
Under assumptions (i)--(iii), the linearized operators satisfy the
spectral hypotheses of the Grillakis--Shatah--Strauss theory
\cite{r19,r19a}.
Condition \eqref{eq:VK} is precisely the Vakhitov--Kolokolov slope
condition.

Therefore the constrained Hessian of the action functional is
positive on the subspace orthogonal to the symmetry direction.
The action functional $S_\omega$ acts as a Lyapunov functional in a
neighborhood of the orbit

\[
\mathcal O(\phi_\omega)
=
\{
e^{i\theta}\phi_\omega :
\theta\in\mathbb R
\}.
\]

Since both energy and mass are conserved by the DNLS flow,
solutions starting sufficiently close to
$\mathcal O(\phi_\omega)$ remain close to this orbit for all times.
Hence the standing wave is orbitally stable.
\end{proof}

\begin{remark}
A complete verification of the spectral assumptions of
Theorem~\ref{thm:conditional_stability} for the present saturable
DNLS model lies beyond the scope of this work.
The theorem should therefore be viewed as placing the localized
ground states constructed in the previous sections within the
classical Grillakis--Shatah--Strauss framework.
\end{remark}

Stability properties of discrete solitons in DNLS lattices have been
extensively investigated in the mathematical physics literature; see,
for example, 
\cite{PelinovskyKevrekidisFrantzeskakis2005} and the references therein.
The above result indicates that the localized ground states obtained
through the variational construction are natural candidates for
dynamically robust coherent structures of the two--dimensional
saturable DNLS equation.
In the final section we present numerical computations illustrating
the analytical results and the convergence of periodic ground states
toward the localized infinite--lattice profile.
\section{Numerical Results}
\label{sec10}

The numerical experiments presented in this section illustrate the
analytical results established in
Theorems~\ref{thm:periodic_ground_states}--
\ref{thm:strong_convergence}.
In particular, they provide numerical evidence for the existence of
localized ground states, their exponential spatial decay, and the strong
convergence of periodic approximations toward the infinite--lattice
profile.

The numerical diagnostics presented below are designed to test the
structural properties predicted by the theory, including the
single--bump character of ground states, the exponential spatial decay
of localized solutions, and the strong $\ell^2(\mathbb Z^2)$ convergence
of periodic approximations. The numerical experiments therefore serve
to illustrate and quantitatively support the analytical results.

All computations are performed on square lattices with periodic
boundary conditions. Periodic ground states are computed by minimizing
the discrete action functional on the associated Nehari manifold.
Unless otherwise stated, we fix $\mu=1$ and $\nu=10$.
The lattice size is chosen sufficiently large so that boundary effects
are negligible in the core region of the solutions.

We begin by validating the convergence of periodic ground states to
localized $\ell^2(\mathbb Z^2)$ solutions as the lattice period
increases. For increasing lattice sizes $N\times N$, we compute the
periodic ground state $u_N$ and compare it to a reference solution
$u_{\mathrm{ref}}$ obtained on a very large lattice.
Figure~\ref{fig:convergence} displays the relative $\ell^2$ error
$\|u_N-u_{\mathrm{ref}}\|_2/\|u_{\mathrm{ref}}\|_2$ as a function of $N$.
The error decays rapidly with $N$, providing numerical evidence for the
strong convergence result stated in Theorem~\ref{thm:strong_convergence}.

\begin{figure}[!t]
\centering
\includegraphics[width=0.50\textwidth]{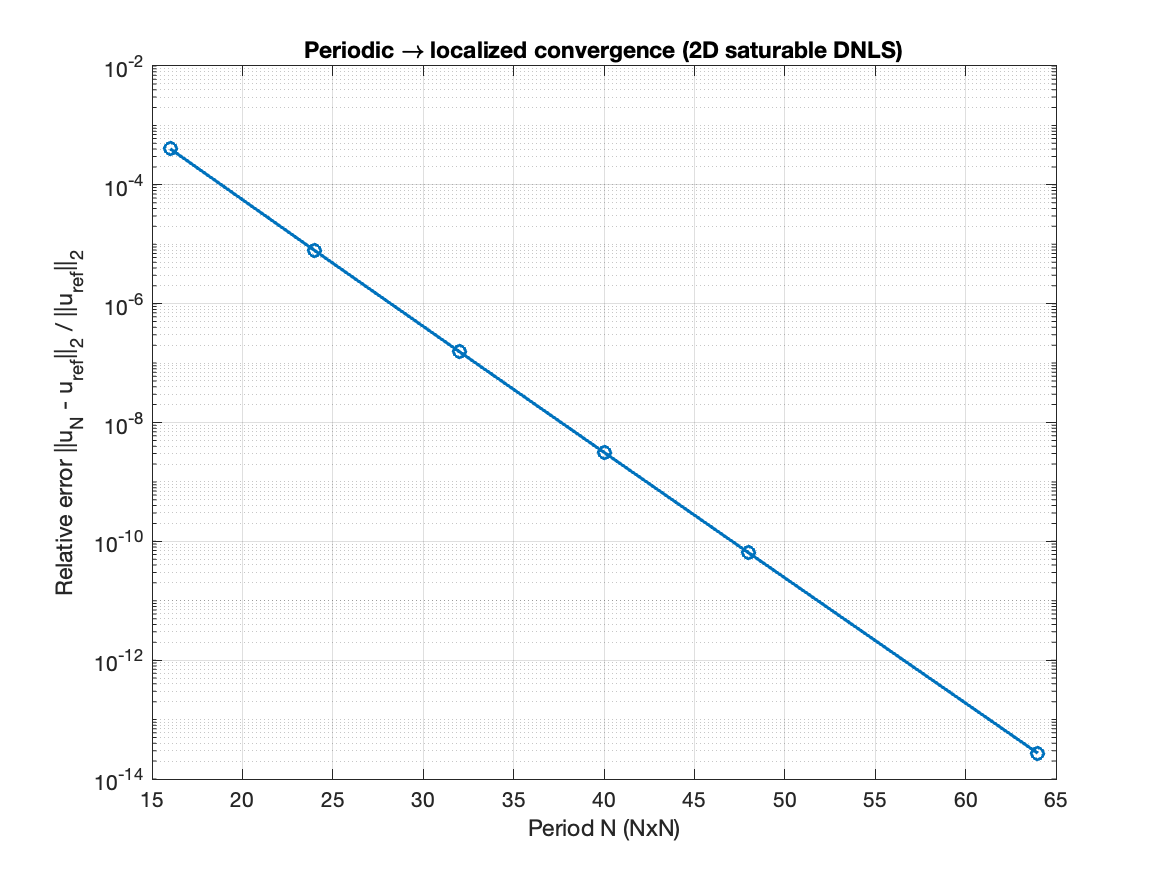}
\caption{Convergence of periodic ground states to a localized solution.
Relative $\ell^2$ error $\|u_N-u_{\mathrm{ref}}\|_2/\|u_{\mathrm{ref}}\|_2$
as a function of the period $N$.}
\label{fig:convergence}
\end{figure}
\FloatBarrier

Next, we investigate the dependence of the ground state branch on the
frequency parameter $\omega$.
Figure~\ref{fig:branch} shows the $\ell^2$ norm $\|u\|_2$ and the
maximal amplitude $\max_{n,m}|u_{n,m}|$ as functions of $\omega$ for a
fixed lattice size $N=96$.
As $\omega$ approaches the spectral edge $\omega=0$ from below, the
solutions become broader and their amplitude decreases. For more
negative values of $\omega$, the solutions become increasingly localized
and have larger peak amplitude. This behavior is consistent with the
variational structure of the problem and the saturable character of the
nonlinearity.
\begin{figure}[!t]
\centering
\includegraphics[width=0.68\textwidth]{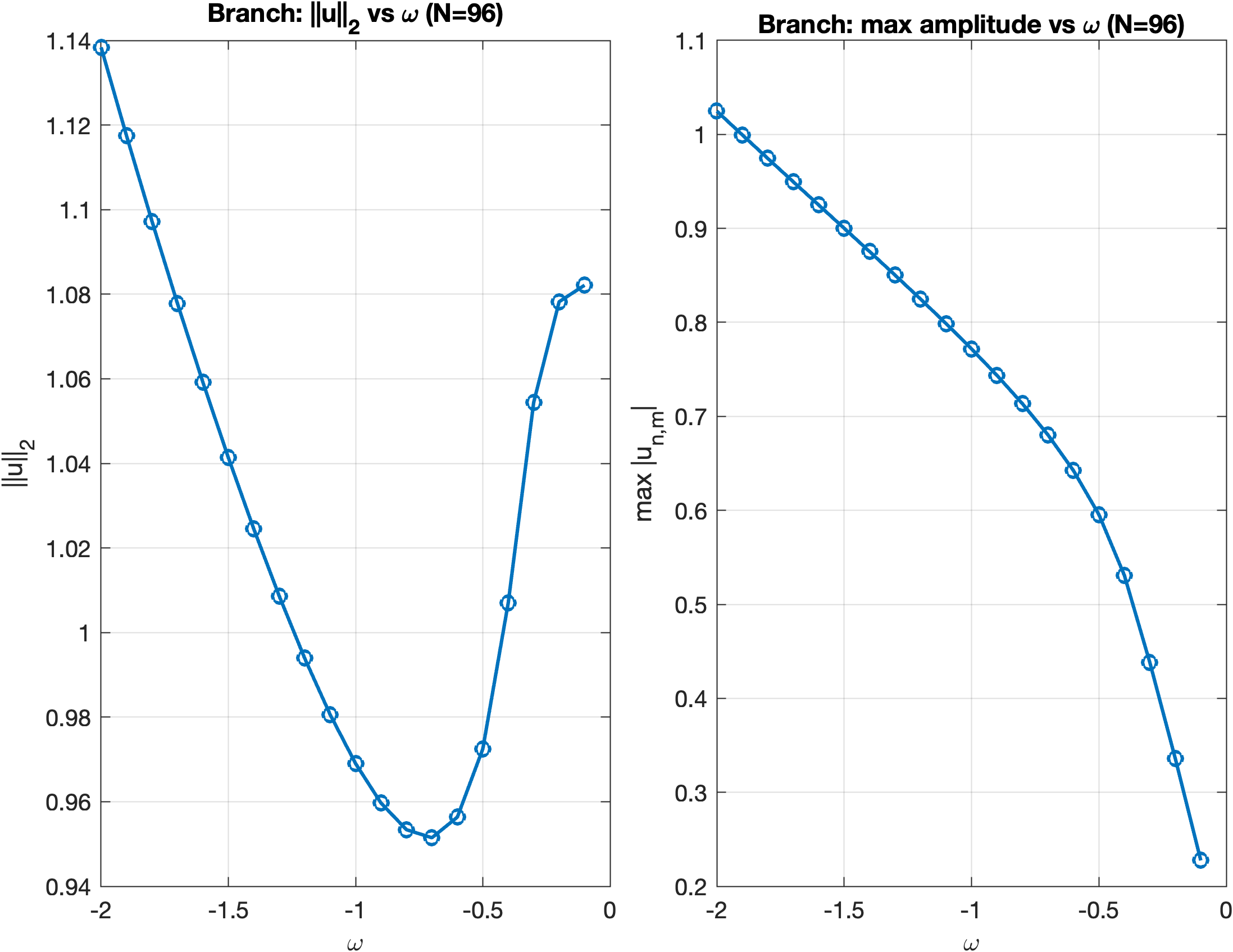}
\caption{Ground state branch for $N=96$.
Left: $\ell^2$ norm $\|u\|_2$ versus $\omega$.
Right: maximal amplitude $\max_{n,m}|u_{n,m}|$ versus $\omega$.}
\label{fig:branch}
\end{figure}

\begin{figure}[!t]
\centering
\includegraphics[width=0.98\textwidth]{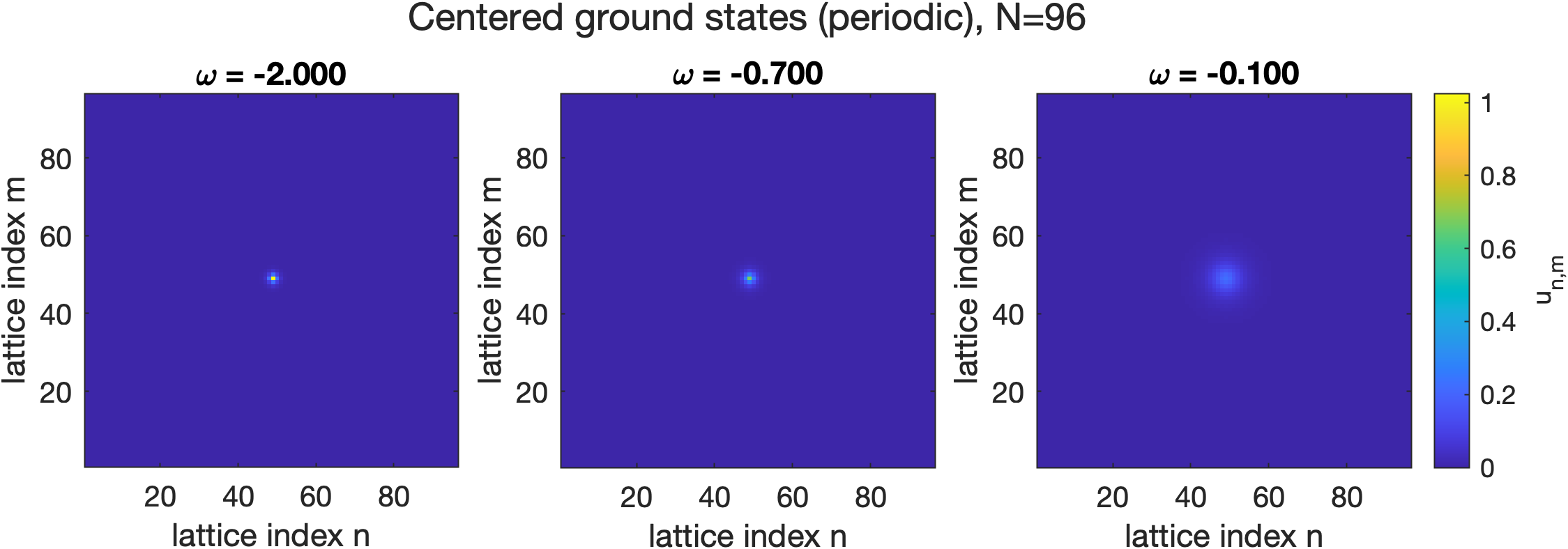}
\caption{Centered periodic ground states on a $96\times96$ lattice at
$\omega=-2.0,-0.7,-0.1$.}
\label{fig:heatmaps}
\end{figure}

\FloatBarrier

The spatial structure of the ground states is illustrated in
Figure~\ref{fig:heatmaps}, where centered periodic ground states are
shown for representative values of $\omega$.
In all cases the solutions exhibit a single--site dominant peak
surrounded by rapidly decaying tails, confirming the single--bump
structure predicted by the variational analysis. As $\omega$ decreases,
the localization becomes stronger and the solution concentrates more
sharply around its center.

To further illustrate the geometry of the solutions,
Figure~\ref{fig:surface} presents a three--dimensional surface plot of
the ground state for $\omega=-0.1$.
The profile exhibits a sharply localized peak with rapid spatial decay
away from the center, providing a direct visualization of localization
on the lattice.

\begin{figure}[!t]
\centering
\includegraphics[width=0.55\textwidth]{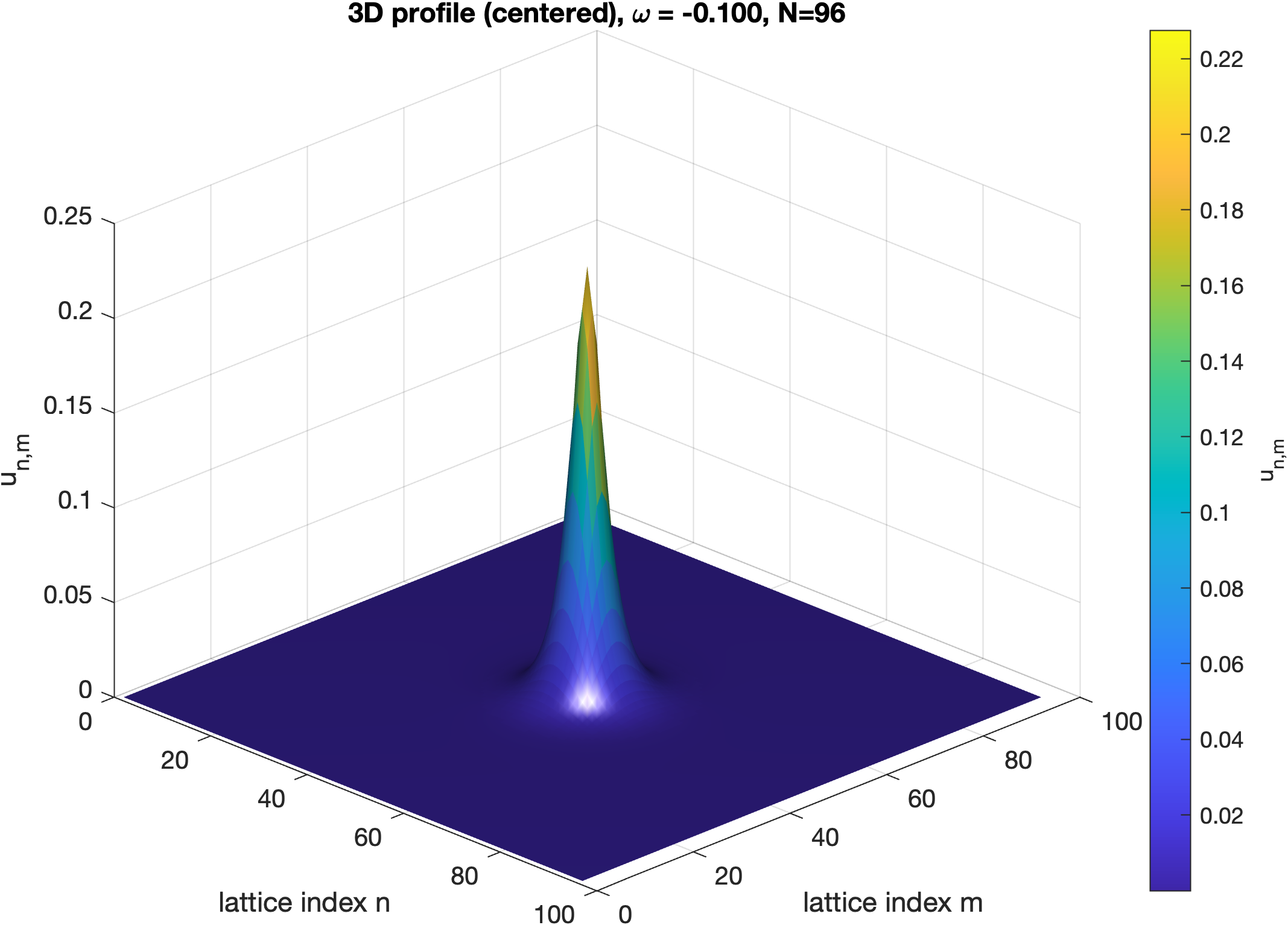}
\caption{Three--dimensional surface plot of the centered ground state
for $\omega=-0.1$ on a $96\times96$ lattice.}
\label{fig:surface}
\end{figure}
\FloatBarrier

A more refined numerical inspection reveals additional qualitative
properties of the computed ground states. In particular, all solutions
along the ground state branch exhibit a clear single--bump structure:
after suitable lattice translations, each solution possesses a unique
dominant maximum located at the center of the lattice, while any
secondary local maxima are several orders of magnitude smaller. This
observation is consistent with the variational characterization of
ground states as energy minimizers on the Nehari manifold.

We also examine the spatial decay of the ground states by constructing
radial profiles centered at the dominant peak and plotting their
magnitude on a logarithmic scale. The resulting curves exhibit an
approximately linear behavior in the far field, indicating exponential
decay. A least--squares fit of the tails yields decay rates that depend
smoothly on $\omega$, with stronger localization corresponding to more
negative values of $\omega$.

\begin{figure}[!t]
\centering
\includegraphics[width=0.32\textwidth]{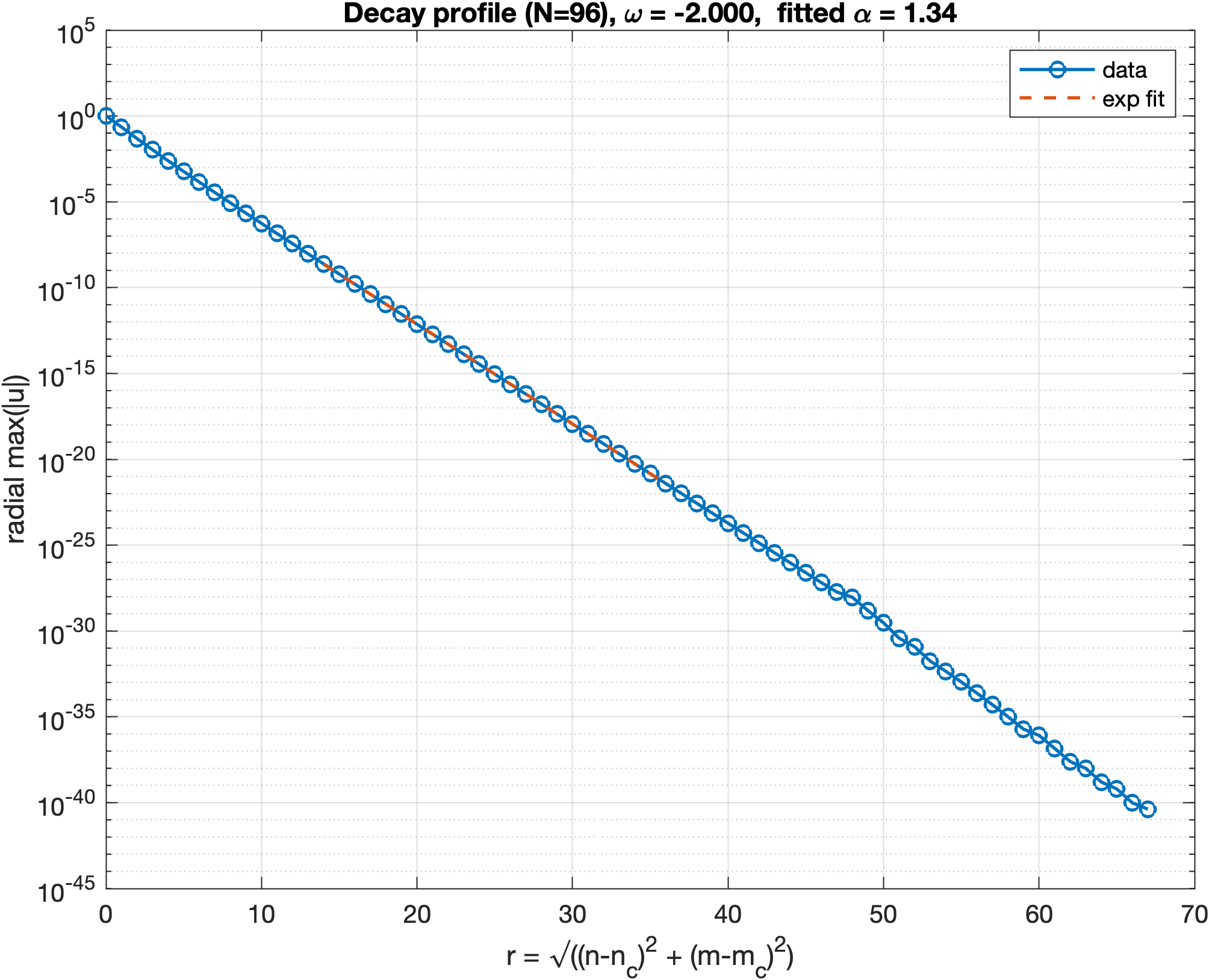}\hfill
\includegraphics[width=0.32\textwidth]{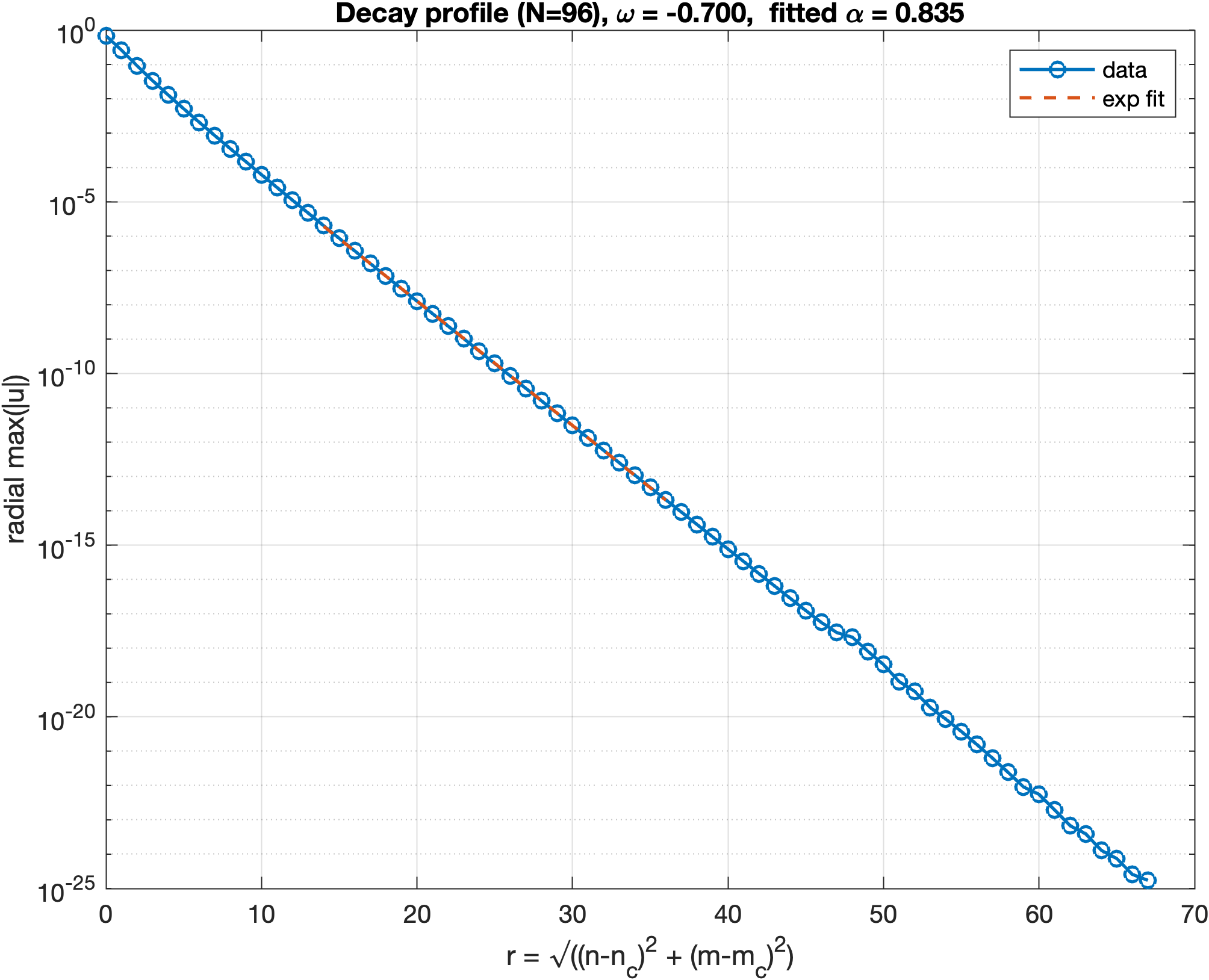}\hfill
\includegraphics[width=0.32\textwidth]{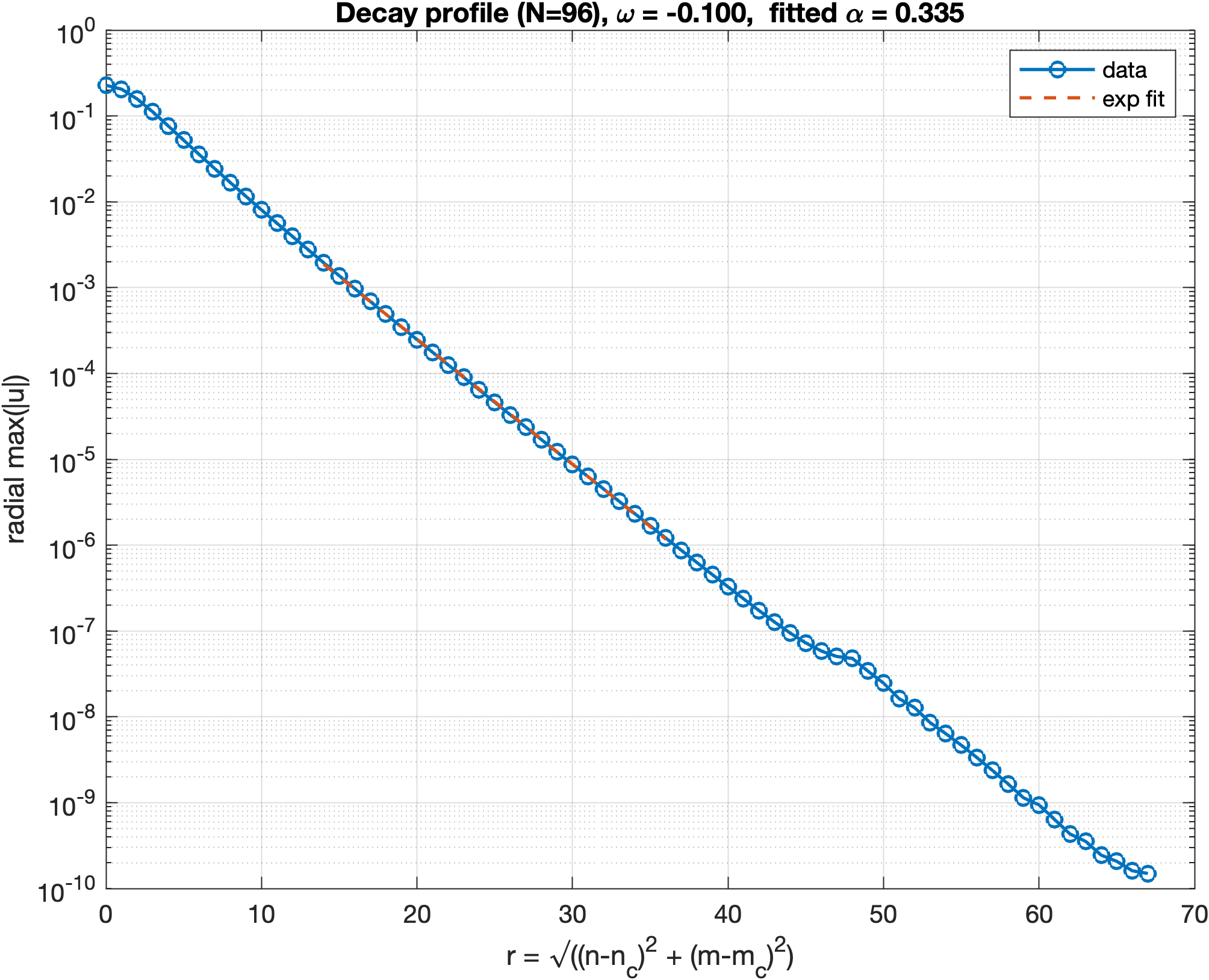}
\caption{Radial decay profiles of centered ground states for
$\omega=-2.0$, $\omega=-0.7$, and $\omega=-0.1$.
Profiles are plotted on a logarithmic scale; dashed lines indicate
least--squares exponential fits in the far field.}
\label{fig:decay_profiles}
\end{figure}
\FloatBarrier

\medskip
\paragraph{Remark on uniqueness.}
Although a complete uniqueness theory for the localized ground states
is beyond the scope of the present work, the numerical computations
suggest that the computed ground state profile is unique up to lattice
translations and the sign symmetry of the equation. In all numerical
experiments performed, the variational minimization consistently
produces a single dominant localized peak after appropriate lattice
translations.

Finally, we investigate the convergence of periodic ground states
toward the infinite--lattice solution using additional diagnostics.
Figure~\ref{fig:SC_diagnostics} shows three indicators as functions of
the lattice period $N$: the Nehari residual, the relative $\ell^2$
error with respect to a reference solution, and the relative mismatch
of the nonlinear term. All quantities decay rapidly as $N$ increases,
confirming that the periodic approximations converge strongly to a
localized ground state and that boundary effects become negligible once
the lattice size exceeds the intrinsic localization length of the
solution. These diagnostics provide a numerical counterpart of the
strong convergence result established in
Theorem~\ref{thm:strong_convergence}.

\begin{figure}[!t]
\centering
\begin{subfigure}[t]{0.32\textwidth}
\centering
\includegraphics[width=\textwidth]{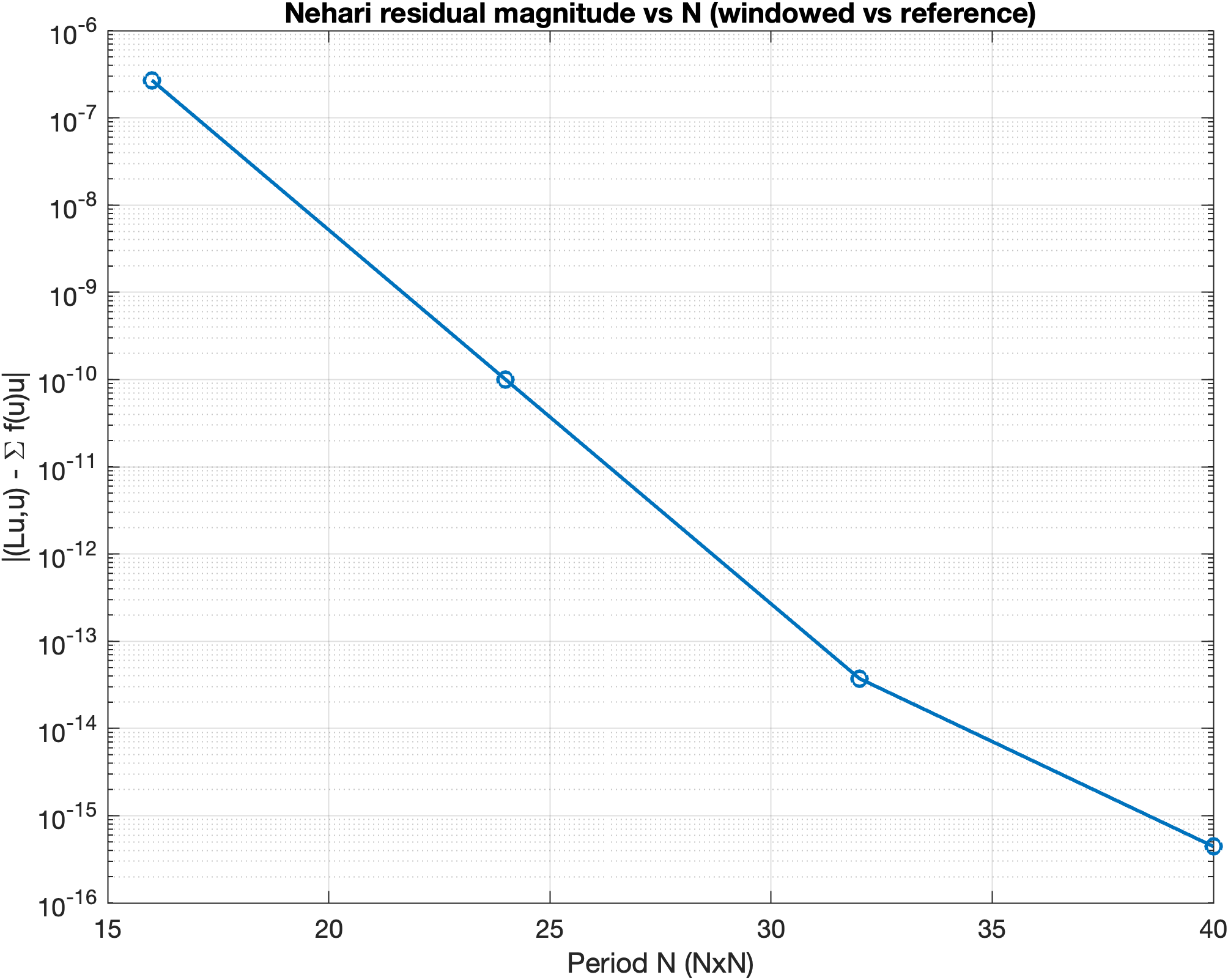}
\caption{Nehari residual}
\label{fig:nehari_residual}
\end{subfigure}\hfill
\begin{subfigure}[t]{0.32\textwidth}
\centering
\includegraphics[width=\textwidth]{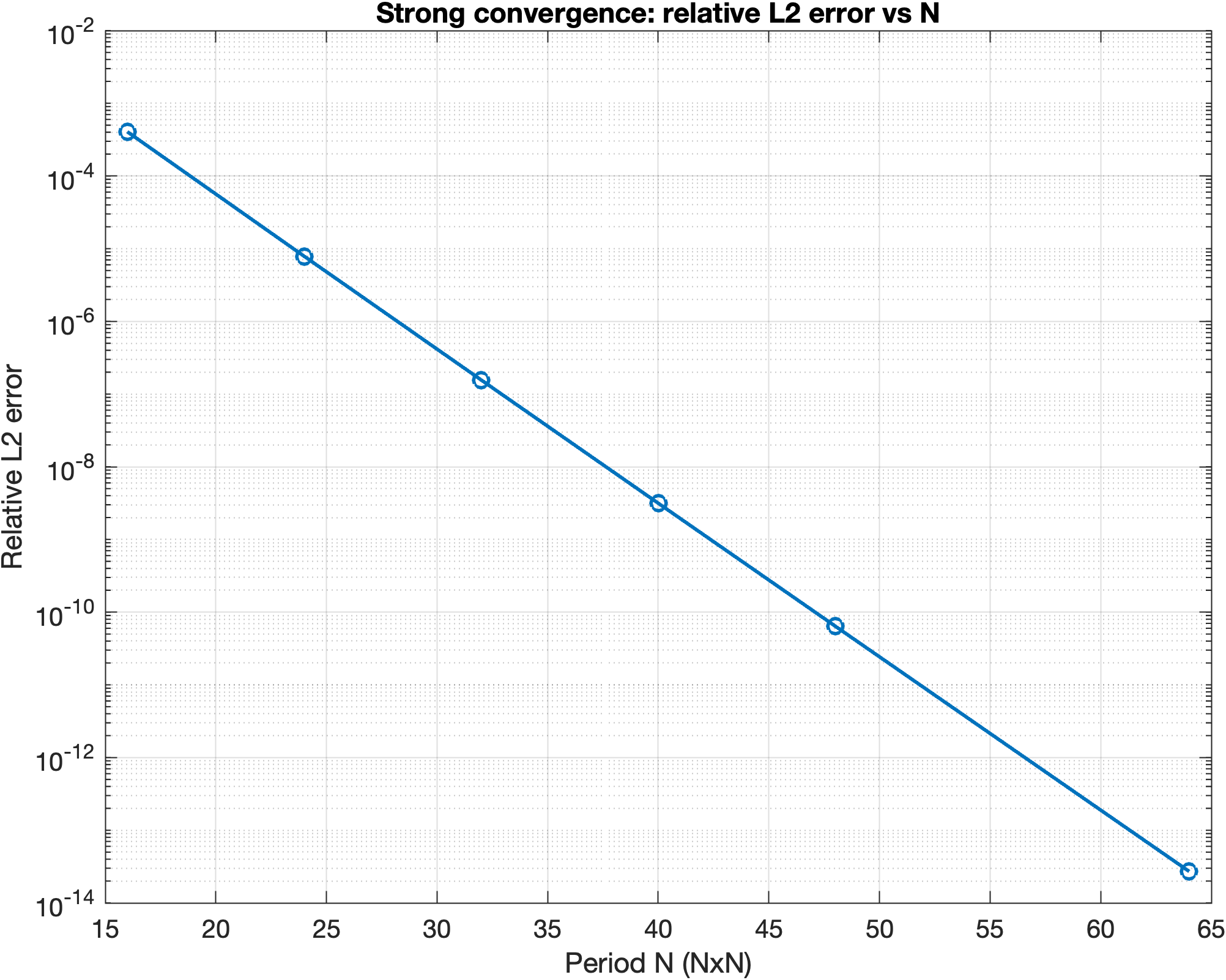}
\caption{Relative $\ell^2$ error}
\label{fig:strong_L2}
\end{subfigure}\hfill
\begin{subfigure}[t]{0.32\textwidth}
\centering
\includegraphics[width=\textwidth]{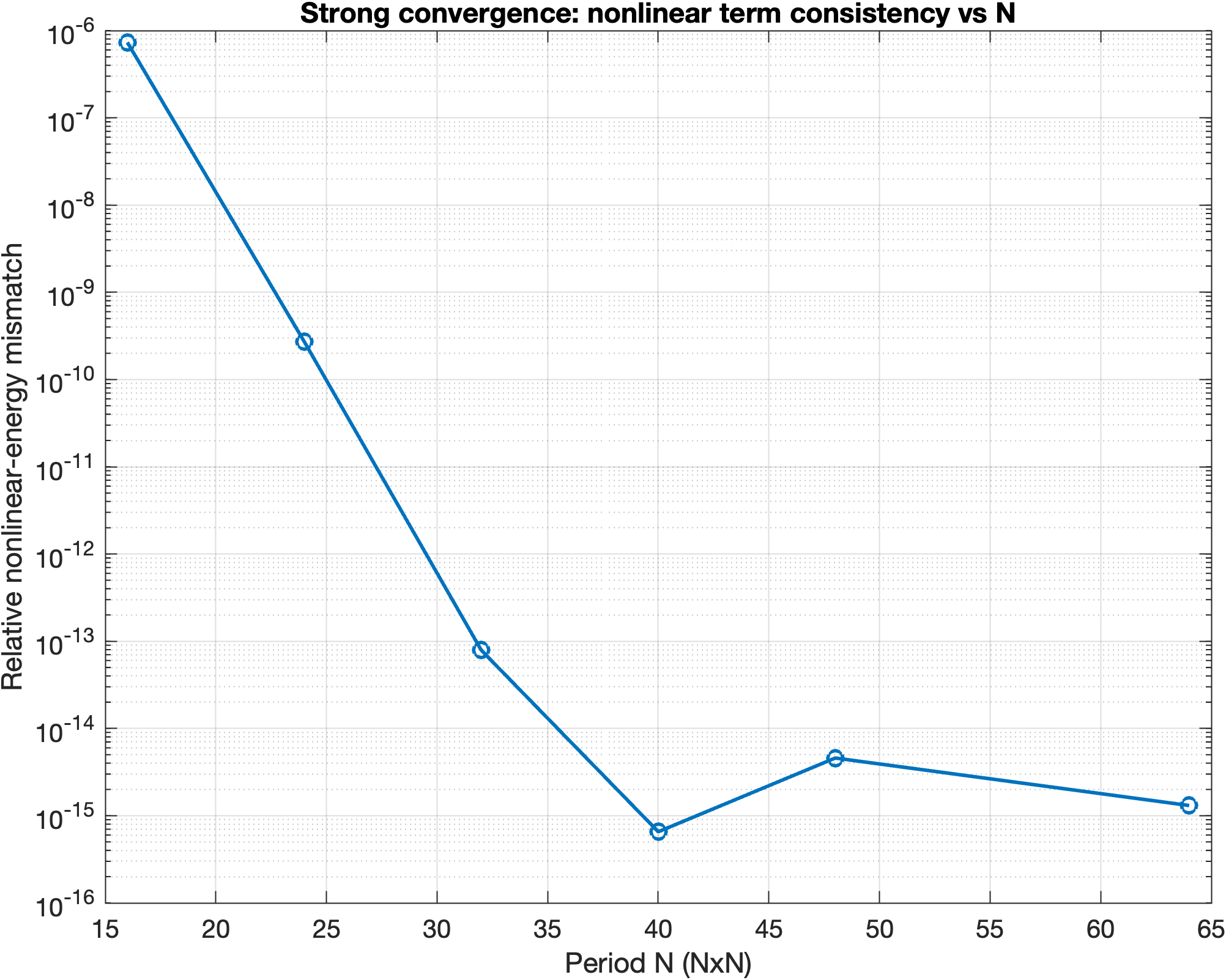}
\caption{Nonlinear mismatch}
\label{fig:nonlinear_consistency}
\end{subfigure}
\caption{Diagnostics for the periodic-to-localized limit.
Left: Nehari residual versus period $N$.
Middle: relative $\ell^2$ error versus $N$.
Right: relative mismatch of the nonlinear term versus $N$.}
\label{fig:SC_diagnostics}
\end{figure}
\FloatBarrier

Taken together, the numerical experiments provide a computational
illustration of the analytical theory developed in this work.
They confirm the existence of strongly localized ground states,
their exponential spatial decay, and the convergence of periodic
ground states toward the localized infinite--lattice profile.
The numerical results are therefore consistent with
Theorems~\ref{thm:periodic_ground_states}--
\ref{thm:strong_convergence}.
\section{Conclusions}\label{sec11}

In this work we developed a variational framework for the
two--dimensional discrete nonlinear Schr\"odinger equation with
saturable nonlinearity on the lattice $\mathbb Z^2$.
The analysis extends the one--dimensional theory of
Pankov and Rothos to a genuinely two--dimensional setting,
where the larger spectral band of the discrete Laplacian and
the lack of compactness associated with lattice translations
require additional arguments.

Using a Nehari manifold approach, we proved the existence of
nontrivial periodic ground states on finite
$(k_1,k_2)$--periodic lattices and established the existence of
spatially localized ground states on the infinite lattice
$\mathbb Z^2$.
The localized solutions were shown to decay exponentially. 
A central result of the paper is the rigorous connection between
the finite and infinite lattice problems.
We proved that, up to lattice translations, periodic ground
states converge strongly in $\ell^2(\mathbb Z^2)$ to a localized
ground state as the lattice periods tend to infinity.
The proof combines variational methods, concentration--compactness
arguments, and a Brezis--Lieb type decomposition which excludes
energy splitting and guarantees convergence to a single localized
profile.
We also placed the resulting standing waves within the classical
Grillakis--Shatah--Strauss framework by deriving a conditional
orbital stability result under the standard spectral and slope
assumptions. This provides a natural connection between the
variational characterization of ground states and their dynamical
behavior under the DNLS flow.
The analytical results were complemented by numerical computations
which illustrate the single--bump structure of the ground states,
their exponential localization, and the convergence of periodic
approximations toward the localized infinite--lattice profile.
The numerical observations are in full agreement with the
theoretical predictions.

The approach developed here combines variational methods,
spectral analysis, and concentration--compactness techniques in a
form that is applicable to a broader class of discrete nonlinear
equations with asymptotically linear nonlinearities.
Possible directions for future work include the study of
higher--dimensional lattices, more general saturable nonlinearities,
and the existence and stability of multi--site or vortex--type
localized states.

\medskip
\noindent\textbf{Data availability.}
No experimental data were generated for this study. The numerical
simulations were performed for illustrative purposes and the corresponding
data are available from the author upon request.

\bibliographystyle{plain}
\bibliography{references_2D-DNLS}

\end{document}